\documentclass[12pt,reqno]{amsart}

\usepackage{amssymb, amsmath, amsfonts, latexsym}
\usepackage{enumerate}
\usepackage{color}
\newtheorem{thm}{Theorem}[]
\newtheorem{lem}{Lemma}[section]

\newtheorem{prop}{Proposition}[]

\newtheorem{rmk}{Remark}[section]
\theoremstyle{definition}

\newtheorem{assumption}{Assumption}[section]
\numberwithin{equation}{section} \theoremstyle{remark}

\title[Universality of convergence rate ]{\bf Convergence rate of extreme eigenvalue of Ginibre ensembles to Gumbel distribution.}

\author{Xinchen Hu}
\address{Xinchen Hu\\ School of Mathematical Sciences $\&$ Laboratory of Mathematics and Complex Systems of Ministry of Education, Beijing Normal University, 100875 Beijing, China.}
\email{huxc@mail.bnu.edu.cn}

\author{Y{utao} Ma}
\address{Yutao MA\\ School of Mathematical Sciences $\&$ Laboratory  of Mathematics $\&$ Complex Systems of Ministry of Education, Beijing Normal University, 100875 Beijing, China.} 
\thanks{The research of Yutao Ma was supported in part by NSFC 12171038,  12571149 and 985 Projects.}
\email{mayt@bnu.edu.cn}

\begin{document}
	\maketitle
	
	\begin{abstract}
		Let $X$ be a real $(\beta=1)$ or complex $(\beta=2)$ Ginibre ensemble.  Let $\{\sigma_i\}_{1\le i\le n}$ be the eigenvalues of $X,$ and $Z_n$ be some rescaled version of $\max_i \Re \sigma_i.$ It was proved that $Z_n$ converges weakly to the Gumbel distribution $\Lambda_{\beta}$ with distribution function $e^{-\frac{\beta}{2}e^{-x}}.$ We further prove that  
		$$\sup_{x\in \mathbb{R}}|\mathbb{P}(Z_n \leq x)-e^{-\frac{\beta}{2}e^{-x}}|=\frac{25\log \log n}{4e \log n}(1+o(1))$$
		and 
		$$ W_1\left(\mathcal{L}(Z_n), \Lambda_{\beta}\right)=\frac{25\log \log n}{4\log n}(1+o(1))$$	for sufficiently large $n$, where $\mathcal{L}(Z_n)$ is the distribution of $Z_n$ and $W_1$ is the Wasserstein distance. Similar results hold for the spectral radius $\max_{i} |\sigma_i|.$ Furthermore, the convergence rates of the complex Ginibre ensemble are universal for complex i.i.d. random matrices under certain moment conditions on entries.  
	\end{abstract} 
	
	
	{\bf AMS Classification Subjects 2020:} 60G70, 60B20, 60B10

	\section{Introduction}

	Let $X$ be an $n\times n$ matrix with independent and identically distributed (i.i.d.) entries $x_{ij} \stackrel{\text { d }}{=} n^{-1 / 2} \xi$, for some complex or real random variable $\xi$ of mean zero and variance one. When $\xi$ is a complex/real Gaussian random variable, $X$ is called a  complex/real Ginibre ensemble. As the simplest model of non-Hermitian random matrices, the complex Ginibre ensemble plays a fundamental role in theoretical studies and serves as a cornerstone for understanding more complex non-Hermitian systems. Ginibre \cite{Ginibre65} gave the explicit density formula of the complex Ginibre ensemble and proved that its eigenvalues form a determinantal point process (DPP). Forrester and Nagao \cite{ForresterNagao} proved that the eigenvalues of the real Ginibre ensemble also form a Pfaffian point process and later Borodin and Sinclair \cite{Borodin09, Borodin16} gave Pfaffian formulas for the more complicated real case.  Here, we recommend \cite{ByunForrester} for more details around Ginibre ensembles.  The readers are also inferred to Akemann and Phillips \cite{Akemannand2012} for further introduction to correlation kernels.  
	
	Girko \cite{Girko84} proved that the empirical spectral distribution (ESD) converges weakly to the uniform distribution on the unit disk in the complex plane, known as the circular law (in contrast to the semicircle law for Hermitian case). Key contributions to the circular law include \cite{Bai97,Bai93i,Bai93ii,Bai24,Bai88,TaoVu08}. Notably, Tao and Vu \cite{TaoVu09} established a universality principle: the limiting distribution of the ESD remains invariant under arbitrary distributions of $\xi,$ provided mild regularity conditions are satisfied.
	
	Now we review the history on limits of extreme statistics. Denote the eigenvalues of $X$ by $\{\sigma_i\}_{i=1}^{n}$ and set $$\aligned 
	\gamma_n&=\frac{1}{2}\left(\log{n}-5\log{\log{n}}-\log({2\pi^4})\right); \\
	\gamma_n^{\prime}&=\log{n}-2\log{\log{n}}-\log{2\pi}. \endaligned 	
	$$

	For the complex Ginibre ensemble, Rider \cite{Rider03} proved that $$\max_{i}|\sigma_i|=1+\sqrt{\frac{\gamma_n'}{4n}}+\frac{\mathcal{G}_n}{\sqrt{4n \gamma_n'}}$$ and Bender \cite{Bender10} showed that  
	$$\max_{i}\Re\sigma_i=1+\sqrt{\frac{\gamma_n}{4n}}+\frac{\widetilde{\mathcal{G}}_{n}}{\sqrt{4n \gamma_n}},$$ 
	 where both $\mathcal{G}_n$ and $\widetilde{\mathcal{G}}_n$ converge weakly to the Gumbel distribution $\Lambda$ whose distribution function is $e^{-e^{-x}}.$ For the real  Ginibre ensemble, Rider and Sinclair \cite{Rider14} proved that $$\max_{i} |\sigma_i|=1+\sqrt{\frac{\gamma_n'}{4n}}+\frac{\mathcal{G}_{n, 1/2}}{\sqrt{4n \gamma_n'}},$$  
   where $\mathcal{G}_{n, 1/2}$ converges to $\Lambda_{1/2},$ whose  distribution function is $e^{-\frac{1}{2}e^{-x}}.$ For both real and complex Ginibre ensembles,  Cipolloni, Erd\"{o}s, Schr\"{o}der and Xu. \cite{Cipolloni22Directional} proved that 
	\begin{equation}\label{Xurightmostginibre} \mathbb{P}(\max_{i}\Re\sigma_i\le 1+\sqrt{\frac{\gamma_n}{4n}}+\frac{x}{\sqrt{4n \gamma_n}})=e^{-\frac{\beta}{2}e^{-x}}+O(\frac{(\log\log n)^2}{\log n}).\end{equation} 
		
	Furthermore, for a complex random matrix with i.i.d. entries, Cipolloni, Erd\H{o}s and Xu \cite{Cipolloni23Universality} 
	demonstrated that  
	\begin{equation}\label{Xuuniversal}
		\lim_{n\to\infty}\mathbb{P}(\max_{i}|\sigma_i|\le 1+\sqrt{\frac{\gamma_n'}{4n}}+\frac{x}{\sqrt{4n \gamma_n'}})=e^{-e^{-x}}
	\end{equation} 
	and \begin{equation}\label{Xuuniversal1} \lim_{n\to\infty}\mathbb{P}(\max_{i}\Re\sigma_i\le 1+\sqrt{\frac{\gamma_n}{4n}}+\frac{x}{\sqrt{4n \gamma_n}})=e^{-e^{-x}}\end{equation}
	for any $x\in\mathbb{R}$ under the following assumption
	\begin{assumption}\label{assumption}
		Assume that $\mathbb{E} \xi=0$,  $\mathbb{E}|\xi|^2=1 $ and $\mathbb{E}\xi^2=0$. Furthermore, for any fixed $k\in \mathbb{N}$,  there exists constant $C_k>0 $ such that
		$$
		\mathbb{E}|\xi^k|\leq C_k . 
		$$
	\end{assumption}
	It is remarkable that \eqref{Xuuniversal} and \eqref{Xuuniversal1} verify the universality of the Gumbel distribution in the sense of pointwise convergence, which confirms the conjecture posted by \cite{ChafaiP}. As mentioned in \cite{Cipolloni23Universality} the convergence rate in \eqref{Xurightmostginibre} can be extended to general i.i.d. matrices satisfying {\bf Assumption} \ref{assumption}, using the comparison theorem obtained in \cite[Theorem 4.1]{Cipolloni23Universality}. Building upon the innovative framework of \cite{Cipolloni23Universality}, particularly leveraging the Green function comparison theorem pioneered therein, and the exact convergence rate, if much slower than $n^{-\epsilon}$ for a sufficiently small $\epsilon>0,$ for the complex Ginibre case could be confirmed to be universal for complex i.i.d. random matrices with entries satisfying   {\bf Assumption} \ref{assumption}. Precisely, 
	setting $$Y_n=\sqrt{4n\gamma_n'}(\max_{i}|\sigma_i|-1-\sqrt{\frac{\gamma_n'}{4n}}),$$
	the key ingredient guaranteeing the universality of the Gumbel distribution is the following inequality 
	\begin{equation}\label{comparison}|\mathbb{P}(Y_{n}\le x)-\mathbb{P}^{\rm Gin}(Y_n\le x)|\lesssim n^{-\epsilon},\end{equation}  which is obtained by Theorem 4.1 in \cite{Cipolloni23Universality}.  Here, $\mathbb{P}$ corresponds to i.i.d. case and $\mathbb{P}^{\rm Gin}$ is with respect to the complex Ginibre ensemble. \eqref{comparison} is still true when $Y_n$ is replaced by $Z_n,$ where 
	$$Z_n=\sqrt{4 n \gamma_n}\left[\max_i\Re\sigma_i-1-\sqrt{\frac{\gamma_n}{4 n}}\right] \quad\text{with} \quad \gamma_n=\frac{1}{2}\left(\log{n}-5\log{\log{n}}-\log({2\pi^4})\right). \quad $$
 		
Recently, for the complex Ginibre ensemble, the second author and Meng \cite{MaMeng25}, via the Kostlan's observation (\cite{Kostlan1992}),  obtained that 
	\begin{equation*}\label{Mamengberry}
		\sup_{x\in \mathbb{R}} |\mathbb{P}(Y_n\le x)-e^{-e^{-x}}|=\frac{2\log \log n}{e \log n}(1+o(1))
	\end{equation*} 
	and
	\begin{equation*}
		W_1 (\mathcal{L}(Y_n), \Lambda)=\frac{2\log \log n}{\log n}(1+o(1)).
	\end{equation*}
	Here, $\mathcal{L}(Y_n)$ is the distribution of  $Y_n$ and $W_1(\mathcal{L}(\zeta), \mathcal{L}(\eta))$ is the Wasserstein distance between the distributions of two random variables $\zeta$ and $\eta$ on $\mathbb{R}$ given by  $$W_1(\mathcal{L}(\zeta), \mathcal{L}(\eta))=\int_{-\infty}^{+\infty} |\mathbb{P}(\zeta\le x)-\mathbb{P}(\eta\le x)| dx.$$
	
The exact order of convergence $\frac{\log\log n}{\log n}$ is much lower than $n^{-\epsilon}$, suggesting that the convergence rate for $Y_n$ is likely universal in the i.i.d. case. However, due to methodological limitations, the analysis in \cite{MaMeng25} does not cover the rightmost eigenvalue of Ginibre ensembles or the spectral radius of real Ginibre matrices.

The primary objective of this paper is to bridge this gap. We first establish precise convergence rates for rightmost eigenvalues of Ginibre ensembles. 	
	\begin{thm} \label{TH1}
		Let $X$ be a  real or complex Ginibre ensemble  and let $\sigma_1,\cdots ,\sigma_n$  be the same as above.	 
		Let 
		$$Z_n=\sqrt{4 n \gamma_n}\left[\max_i\Re\sigma_i-1-\sqrt{\frac{\gamma_n}{4 n}}\right] \quad\text{with} \quad \gamma_n=\frac{1}{2}\left(\log{n}-5\log{\log{n}}-\log({2\pi^4})\right). \quad $$
		Then, for  sufficiently large $n$, we have		$$\aligned
		\sup_{x\in \mathbb{R}}|\mathbb{P}(Z_n\le x)-e^{-\frac{\beta}{2}e^{-x}}|&=\frac{25\log \log n}{4e\log n}(1+o(1)); \\
		W_{1}\left(\mathcal{L}(Z_n),\Lambda_{\beta}\right)&=\frac{25\log \log n }{4\log n}(1+o(1)).
		\endaligned $$  
	\end{thm}
Now we state the parallel one for the spectral radius. 
	\begin{thm}\label{TH2}
		Let $X$ be a  real or complex Ginibre ensemble  and let $\sigma_1,\cdots ,\sigma_n$  be the same as above. Let 
		$$Y_n=\sqrt{4 n \gamma_n^{\prime}}[\max_i\left|\sigma_i\right|-1-\sqrt{\frac{\gamma_n^{\prime}}{4 n}}] \quad\text{with} \quad \gamma_n^{\prime}=\log{n}-2\log{\log{n}}-\log{2\pi}. \quad $$	
		Then, we have	
				$$\aligned
		\sup_{x\in \mathbb{R}}|\mathbb{P}(Y_n\le x)-e^{-\frac{\beta}{2}e^{-x}}|&=\frac{2\log \log n}{e\log n}(1+o(1));\\
		W_{1}\left(\mathcal{L}(Y_n),\Lambda_{\beta}\right)&=\frac{2\log \log n }{\log n}(1+o(1))
		\endaligned $$
		for  sufficiently large $n.$
 	\end{thm}
Hereafter, we use the same notations $Y_n$ and $Z_n$  for both real and complex Ginibre ensembles, provided that no confusion arises and we  use $\mathbb{P}^{\rm Gin}$ only when complex i.i.d. case appears.

	
	We will provide detailed proofs of Theorem \ref{TH1} and Theorem \ref{TH2}, then outline how the Green function comparison theorem (introduced in \cite[Theorem 4.1]{Cipolloni23Universality}) extends exact convergence results to general complex IID case.  For extreme eigenvalue statistics of non-Hermitian random matrices, the correlation kernel has been widely applied(see \cite{AB10, AP14, Bender10, Borodin09, Rider14, Cipolloni22Directional}, among others).  In particular, for the rightmost eigenvalues of Ginibre ensembles, Cipolloni, Erd\"{o}s, Schr\"{o}der, and Xu  \cite{Cipolloni22Directional} presented a concise proof based on estimates of the correlation kernel and derived a convergence rate of  $O\big(\frac{(\log\log n)^2}{\log n}\big)$.  Here, building upon the framework of \cite{Cipolloni22Directional},  we refine their analysis with sharper estimates to establish Theorem \ref{TH1} and we use the same method to derive Theorem \ref{TH2}.  
	
	Hereafter, we use frequently $t_n=O(z_n)$ or $t_n=o(z_n)$ if $\lim_{n\to\infty}\frac{t_n}{z_n}=c\neq 0$  or $\lim_{n\to\infty}\frac{t_n}{z_n}=0$. For any $f,g\geq 0$, we write $f\lesssim g$ if $f\leq Cg$ for some constant $C>0$. We also use $t_n\ll z_n$ or equivalently $z_n\gg t_n$ to represent $\lim_{n\to\infty}\frac{t_n}{z_n}=0.$ 
	
	For the clarity of the presentation of the proof, we outline the sketch of the proofs of Theorems \ref{TH1} and \ref{TH2}. 
	For both Berry-Esseen bound and $L^1$-Wasserstein distance, we cut $\mathbb{R}$ as the union of three intervals $$\mathbb{R}=(-\infty, -\frac14\log\log n)\cup [-\frac14\log\log n, (\log n)^{1/4}]\cup ((\log n)^{1/4}, +\infty)$$ and we will demonstrate that the precise convergence rates are attained on the interval $[-\frac{1}{4}\log\log n,\ (\log n)^{1/4}]$, while the remaining two terms are negligible. Given certain terms $W_n^{(t)}$, which depends on the correlation kernel and the corresponding indicator function related to $Z_n$ (to be specified later), one gets  
	$$\aligned \mathbb{P}(Z_n\le t)={\rm det}(1-W_n^{(t)}).\endaligned$$ Moreover, it has been verified that 
	$$
		|\det(1-W_n^{(t)})-\exp(-\operatorname{Tr}(W_n^{(t)}))|\leq \|W_n^{(t)}\|_2 \exp\big\{\frac12(\|W_n^{(t)}\|_2+1)^2-\operatorname{Tr}(W_n^{(t)})\big\}
	$$ uniformly on $|t|\lesssim (\log n)^{1/4}.$ Here $\|W_n^{(t)}\|_2$ and $\operatorname{Tr}(W_n^{(t)})$ will be introduced later. 
	Thereby, once  $\|W_n^{(t)}\|_2\ll \frac{\log \log n}{\log n}$ uniformly on $|t|\lesssim (\log n)^{1/4},$ we see 
	$$|\mathbb{P}(Z_n\le t)-e^{-e^{-t}}|=|e^{-\operatorname{Tr}(W_n^{(t)})}-e^{-e^{-t}}|+o(\frac{\log\log n}{\log n})$$  
	uniformly on $|t|\lesssim (\log n)^{1/4}.$ The core ingredients are the proper upper bounds  on $\|W_n^{(t)}\|_2$  and precise asymptotical  expression of  $\operatorname{Tr}(W_n^{(t)}).$
	Similar treatment will be done for the spectral radius case. We leave the reason for the choice of both $\ell_1(n)$ and $\ell_2(n)$ later.    
	
	Therefore, the next section is devoted to the estimates on $\|W_n^{(t)}\|_2$  and precise asymptotical  expressions of  $\operatorname{Tr}(W_n^{(t)}).$ And the third section focuses on the proofs of Theorems \ref{TH1} and \ref{TH2} and we state the universality of convergence rate for extreme eigenvalues in the forth section.  At last, we introduce new convergences for both $Z_n$ and $Y_n$ of order $\frac{1}{\log n}$ under new scales $\gamma_n$ and $\gamma_n',$ respectively.  
		
	\subsection*{Acknowledgment} The authors wish to extend their sincere gratitude to Professors Forrester and Yuanyuan Xu for their invaluable suggestions on an earlier version of this manuscript. Their insightful comments have significantly improved the clarity and rigor of the presentation, and we are deeply grateful for their contributions to the development of this work.  
	
	\section{Preparation work} 
	This section is devoted to the precise asymptotic on ${\rm Tr}(W_n^{(t)})$ and $\|W_n^{(t)}\|_2$ and other similar terms, which is mainly based on the framework of \cite{Cipolloni22Directional} 
	and the readers are inferred to \cite{Cipolloni22Directional} for further technical details. 
	
We first work on the complex Ginibre ensemble. 
\subsection{Correlation kernel of Complex Ginibre}

	It was proved in \cite{Rider03} that the eigenvalues of a complex Ginibre matrix has joint probability density as follows 
	\begin{equation*}\rho_n(\boldsymbol{z})=\rho_n(z_1,\ldots,z_n):=\frac{n^n}{\pi^n1!\cdots n!}\exp\left(-n\sum_i\lvert z_i\rvert^2\right)\prod_{i<j}\left(n\lvert z_i-z_j\rvert^2\right)\end{equation*}
	and  Ginibre in \cite{Ginibre65} concludes that these eigenvalues form a determinantal point process with 
	$$
	\rho_n(\boldsymbol{z})=\frac{n^n}{\pi^nn!}e^{-n|\boldsymbol{z}|^2}\det\left(K_n\left(z_i,z_j\right) \right)_{i,j=1}^n,
	$$
	where the correlation kernel is given by $$K_n(z,w):=\sum_{l=0}^{n-1}\frac{(nz\bar{w})^l}{l !}$$ with $\bar{w}$ being the complex conjugate of $w.$ An important formula says 
	\begin{equation}\label{expressKn} 
		K_n(z,w)=e^{n z\bar{w}}\frac{\Gamma(n,nz\bar{w})}{\Gamma(n)},
	\end{equation}
	where, $\Gamma$ is the Gamma function and $\Gamma(s, z)$ is the incomplete Gamma function which is given by $$
	\Gamma(s,z) =\int_z^{\infty}t^{s-1}e^{-t} \mathrm{~ d} t.
	$$
	
	We now present several conclusions regarding $K_n,$ which will be instrumental in the proof of the complex Ginibre. The first lemma is borrowed from \cite{Cipolloni22Directional}.
	\begin{lem}\label{lemma1}
		Given a function $g:\mathbb{C}\to[0,1]$ and let $K_n$ be defined as above.  Define  $$
		\widetilde{K}_n(z,w)=\frac{n}{\pi}e^{-n\frac{|z|^2+|w|^2}{2}}K_n(z,w).
		$$ Then, 
		$$ 
		\mathbb{E}\prod_{i=1}^{n}(1-g(\sigma_i))=\det \left(1-\sqrt{g} \widetilde{K}_n \sqrt{g}\right).
		$$
		
	\end{lem}

	\subsubsection{Estimates required for Rightmost eigenvalues}
	
	Recall  $$
	Z_{n}=\sqrt{4n \gamma_n}\big[\max_i \Re  \sigma_i-1-\sqrt{\frac{\gamma_n}{4n}}\big].
	$$
	 
	Next, we connect the probability $\mathbb{P}(Z_{n} \leq t)$ with $\sqrt{g}\widetilde{K}_n \sqrt{g}$ for some particular function $g.$ For that aim, we choose $g=\chi_{A(t)}^{}$ with the set $A(t)$ being defined as \begin{equation}\label{A_t rightmost}
			A(t):=\left\{z \in \mathbb{C} \left\lvert\, \Re z \geq 1+\sqrt{\frac{\gamma_n}{4 n}}+\frac{t}{\sqrt{4 \gamma_n n}}\right.\right\}
	\end{equation}and set $W_n^{(t)}=\sqrt{\chi_{A(t)}^{}}\widetilde{K}_n\sqrt{\chi_{A(t)}^{}}$ for simplicity.  
	For any $t\in \mathbb{R}$, we have 
	$$\mathbb{P}(Z_{n} \leq t)=\mathbb{P}\big(\cap_{i=1}^n \{\sigma_i \notin A(t)\} \big)=\mathbb{E}\big[\prod_{i=1}^{n}\left[1-\chi_{A(t)}^{} (\sigma_i)\right]\big]=\det(1-W_n^{(t)}),$$ whence, 
	\begin{equation}\label{diff}
		\begin{aligned}
			\big|\mathbb{P}(Z_{n} \leq t)-e^{-e^{-t}}\big|
			&=\big|\det(1-W_n^{(t)}) -e^{-e^{-t}}\big|.\\
		\end{aligned}
	\end{equation}  
	Inspired by previous works on the exact convergence rate of spectral radius both for the Ginibre case and the large chiral non-Hermitian random matrix in \cite{MaMeng25} and \cite{MaWang25}, we would first concentrate our efforts on the most difficult estimate $\big|\mathbb{P}(Z_{n} \leq t)-e^{-e^{-t}}\big|$ for $|t|\lesssim (\log n)^{1/4}.$
	
	The expression (7.11) in \cite{Gohberg} provides an estimate for $\det(1-W_n^{(t)}),$ stating that   
	\begin{equation}\label{E2}
		|\det(1-W_n^{(t)})-\exp(-\operatorname{Tr}(W_n^{(t)}))|\leq \|W_n^{(t)}\|_2 \exp\big\{\frac12(\|W_n^{(t)}\|_2+1)^2-\operatorname{Tr}(W_n^{(t)})\big\},
	\end{equation} 
	where  
	$$\aligned \operatorname{Tr}(W_n^{(t)}):&=\int_{\mathbb{C}} W_n^{(t)}(z, z)\,\mathrm{d}^2 z=\int_{A(t)} \widetilde{K}_n(z, z) d^2 z ; \\
	\|W_n^{(t)}\|_2^2:&=\int_{\mathbb{C}^2}|W_{n}^{(t)}(z, w)|^2\,\mathrm{d} z \mathrm{d} w=\int_{A^2(t)}|\widetilde{K}_n(z, w)|^2\,\mathrm{d} z \mathrm{d} w. \endaligned 
	$$
	It was proved in \cite{Cipolloni22Directional} that 		\begin{equation}\label{quote formula4}
		\left\|W_n^{(t)}\right\|_2 \lesssim e^{-\sqrt{\log n} / 32}
	\end{equation} 
	uniformly on $|t|\leq
	\frac{\sqrt{\log n}}{10}.$ This, with \eqref{E2} and the fact ${\rm Tr}(W_n^{(t)})\ge 0,$ ensures that $$|\det(1-W_n^{(t)})-\exp(-\operatorname{Tr}(W_n^{(t)}))|\lesssim e^{-\sqrt{\log n} / 32}.$$ 
	Therefore, using 
	the triangle inequality, we have  \begin{equation}\label{Keydecom}
		\begin{aligned}
		\left|\mathbb{P}(Z_{n} \leq t)-e^{-e^{-t}}\right|=|e^{-{\rm Tr}(W_n^{(t)})}-e^{-e^{-t}}|+O(e^{-\sqrt{\log n} / 32})			
		\end{aligned}
	\end{equation} 
	and then the precise asymptotic on $\operatorname{Tr}(W_n^{(t)})$ is in need now. 
	
	In \cite{Rider14, Cipolloni22Directional}, they analyze 
	$\widetilde{K}_n(z, z)$ by setting \begin{equation}\label{z} z=1+\sqrt{\frac{\gamma_n}{4n}}+\frac{x}{\sqrt{4\gamma_n n}}+\frac{iy}{(\gamma_n n)^{\frac{1}{4}}}.\end{equation}  
	Next, we first derive an asymptotic expansion of $\widetilde{K_n}(z, z)$  in the specified regime $|x|+y^2\lesssim (\log n)^{1/4},$ where $x, y, z$ are interrelated via \eqref{z}. This serves as our fundamental tool for calculating $\operatorname{Tr}(W_n^{(t)}).$ 
	
	\begin{lem}\label{RY}
		Let $\widetilde{K}_n$ be the same as above and $z$ satisfy \eqref{z}. In the regime $|x|+y^2 \lesssim (\log n)^{1/4},$ we have
		$$
		\frac{\widetilde{K}_n(z, z)}{2(\gamma_n n)^{3 / 4}}
		=\frac{ e^{-x-y^2}}{\sqrt{\pi}}(\frac{\log n}{2\gamma_n})^{5/4}\big(1-\gamma_n^{-1}(1+x+y^2+\frac12(x+y^2)^2)\big)\left(1+O((\log n)^{-1})\right) 
		$$ 
		for sufficiently large $n.$
	\end{lem}
	\begin{proof} It follows from the definition of $\widetilde{K}_n$ and \eqref{expressKn} that 
		\begin{equation}\label{sumexpress}
			\frac{\widetilde{K}_n(z, z)}{2(\gamma_n n)^{3 / 4}}=\frac{n}{2\pi(\gamma_n n)^{3/4}}\frac{\Gamma(n, n|z|^2)}{\Gamma(n)}.
		\end{equation}
		
		We revisit the asymptotic estimate in \cite[Lemma 3.2]{Rider14} expressed as 
		\begin{equation}\label{2.14}
			\frac{\Gamma(n, n|z|^2)}{\Gamma(n)}=\frac{|z|^2\mu(|z|^2)\operatorname{erfc}(\sqrt{n}\mu(|z|^2))}{\sqrt{2}(|z|^2-1)}\left(1+O\left(\frac{1}{n||z|^2-1|}\right)\right),
		\end{equation}  where  $\mu(x):=\sqrt{x-\log x-1}$ for $x>0.$ 
		The asymptotic expansion \begin{equation}\label{erfc}\operatorname{erfc}(z)=\frac{e^{-z^2}}{\sqrt{\pi}z}\left(1-\frac{1}{2z^2}+O\left(\frac{1}{|z|^4}\right)\right)\end{equation} for $|\arg z|<\frac{3\pi}{4}$ implies that 	
		\begin{equation}\label{uerfc}\aligned \mu(|z|^2)\operatorname{erfc}(\sqrt{n}\mu(|z|^2))&=\frac{1}{\sqrt{n\pi}}e^{-n\mu^2(|z|^2)}\left(1-\frac{1}{2n\mu^2(|z|^2)}+O\left(\frac{1}{n^2 \mu^4(|z|^2)}\right)\right).\endaligned \end{equation}
		By definition, 
		\begin{equation}\label{z2express}
			|z|^2=z\bar{z}=1+\sqrt{\frac{\gamma_n}{n}}+\frac{x+y^2}{\sqrt{n\gamma_n}}+\frac{(x+\gamma_n)^2}{4n\gamma_n}
		\end{equation}
		and then $$
		\begin{aligned}
			(|z|^2-1)^2&=(\sqrt{\frac{\gamma_n}{n}}+\frac{x+y^2}{\sqrt{n\gamma_n}}+\frac{(x+\gamma_n)^2}{4n\gamma_n})^2\\
			&=\frac{\gamma_n +2x +2y^2}{n} +\frac{(x+y^2)^2}{n\gamma_n}+O\left(\frac{\gamma_n^{\frac32}}{n^{\frac{3}{2}}}\right).
		\end{aligned}
		$$
		Therefore, together with Taylor's formula $$x-1-\log x=\frac{(x-1)^2}{2}+O((x-1)^3)$$ for $|x-1|$ small enough, one gets  
		$$\mu^2(|z|^2)=|z|^2-1-\log |z|^2=\frac{\gamma_n +2x +2y^2}{2n} +\frac{(x+y^2)^2}{2n\gamma_n}+O\left(\gamma_n^{\frac32}n^{-\frac{3}{2}}\right).$$ 
		Consequently, one obtains by $e^{x}=1+x+O(x^2)$ for $|x|$ sufficiently small that 
		$$\aligned e^{-n\mu^2({|z|^2})}&=e^{-\frac{\gamma_n}{2}-x-y^2-\frac{(x+y^2)^2}{2\gamma_n}} (1+O(n^{-1/2}\gamma_n^{3/2}))\\
		&=e^{-\frac{\gamma_n}{2}-x-y^2}(1-\frac{(x+y^2)^2}{2\gamma_n})\big(1+O(\gamma_n^{-1})\big)\endaligned $$ 	
		since $|x|+y^2\lesssim (\log n)^{1/4}$ and similarly
		$$\aligned 
		\frac{1}{2n\mu^2(|z|^2)}&=\left(\gamma_n+2x+2y^2+(x+y^2)^2\gamma_n^{-1}+O(\gamma_n^{3/2}n^{-3/2})\right)^{-1}\\
		&=\gamma_n^{-1}(1-2\gamma_n^{-1}(x+y^2))\big(1+o(\gamma_n^{-1})\big). 
		\endaligned $$ 
		Plugging these two expressions into the formula \eqref{uerfc}, we get in further that 
		$$
		\begin{aligned}
			\mu(|z|^2)\operatorname{erfc}(\sqrt{n}\mu(|z|^2))&=\frac{1}{\sqrt{n \pi}}e^{-\frac{\gamma_n}{2}-x-y^2}(1-\frac{(x+y^2)^2+2}{2\gamma_n})(1+O(\gamma_n^{-1})).	
		\end{aligned}
		$$
		Inserting this back into \eqref{2.14} leads to the conclusion that
			\begin{align}
			\frac{\Gamma(n,n|z|^2)}{\Gamma(n)}
			&=\frac{|z|^2}{\sqrt{2\pi n}(|z|^2-1)}e^{-\frac{\gamma_n}{2}-x-y^2}(1-\frac{(x+y^2)^2+2}{2\gamma_n})(1+O(\gamma_n^{-1})) \nonumber\\
			&=\frac{1+O(\gamma_n^{-1})}{\sqrt{2\pi\gamma_n}(1+\gamma_n^{-1}(x+y^2))}e^{-\frac{\gamma_n}{2}-x-y^2}(1-\frac{(x+y^2)^2+2}{2\gamma_n}) \nonumber\\
			&=\frac{1+O( \gamma_n^{-1})}{\sqrt{2\pi\gamma_n}}e^{-\frac{\gamma_n}{2}-x-y^2}(1-\frac{\frac12(x+y^2)^2+1+x+y^2}{\gamma_n}). \label{2.12}
		\end{align}
			This, with \eqref{sumexpress}, implies that 
		$$
		\begin{aligned}
			\frac{\widetilde{K}_n(z, z)}{2(\gamma_n n)^{3 / 4}}
			&=\frac{n^{1/4}(1+O(\gamma_n^{-1}))}{(2\pi)^{3/2} \gamma_n^{\frac{5}{4}}}e^{-\frac12\gamma_n-x-y^2}\big(1-\gamma_n^{-1}\big(1+x+y^2+\frac12(x+y^2)^2\big)\big).
		\end{aligned}
		$$
		Now $$\gamma_n=\frac{\log n -5\log \log n -\log (2\pi^4)}{2} $$ and then 
		\begin{equation}\label{equa} e^{-\frac12\gamma_n}=n^{-1/4} (\log n)^{5/4} (2\pi^4)^{1/4}.\end{equation}
		Therefore, 
		\begin{equation*}
			\begin{aligned}
				\quad \frac{\widetilde{K}_n(z, z)}{2(\gamma_n n)^{3 / 4}}
				=\frac{ e^{-x-y^2}}{\sqrt{\pi}}(\frac{\log n}{2\gamma_n})^{5/4}\big(1-\gamma_n^{-1}(1+x+y^2+\frac12(x+y^2)^2)\big)\left(1+O(\gamma_n^{-1})\right).
			\end{aligned}
		\end{equation*}
		The proof is completed. 
	\end{proof}
	
	Next, we provide an estimate for ${\rm Tr}(W_n^{(t)}).$	
	\begin{lem}\label{Trlem}
		Let  $W_n^{(t)}$ be defined as above with $|t| \lesssim (\log n)^{1/4}.$ Then, 
		$$
		\operatorname{Tr} (W_n^{(t)})=e^{-t}\big(1+\frac{c_n-t^2-5t}{\log n})\left(1+O\left((\log n)^{-1}\right)\right), $$
		where $c_n=(25\log \log n+5\log(2\pi^4)-35)/4.$

	\end{lem}
	\begin{proof} 
		Review $$A(t):=\left\{z \in \mathbb{C} \left\lvert\, \Re z \geq 1+\sqrt{\frac{\gamma_n}{4 n}}+\frac{t}{\sqrt{4 \gamma_n n}}\right.\right\}$$
		and $$z=1+\sqrt{\frac{\gamma_n}{4n}}+\frac{x}{\sqrt{4\gamma_n n}}+\frac{iy}{(\gamma_n n)^{\frac{1}{4}}}.$$ It indicates that 
		$z\in A(t)$ if and only if $x\ge t.$
		We set $t_0=4(|t|+(\log n)^{1/4})$ and decompose the integral below into three parts as 
		$$
		\begin{aligned}
			\operatorname{Tr} (W_n^{(t)}) & =\int_{A(t)} \widetilde{K}_n(z, z) \mathrm{d}^2 z \\
			& =\left(\int_t^{t_0} \int_{y^2<2 t_0}+\int_t^{t_0} \int_{y^2 \geq 2 t_0}+\int_{t_0}^{\infty} \int_{\mathbb{R}}\right)  \frac{\widetilde{K}_n(z, z)}{2(\gamma_n n)^{3 / 4}} \mathrm{~d} y \mathrm{~d} x \\
			& =:I_1+I_2+I_3 .\\
		\end{aligned}
		$$
		Next, we give estimates on $I_1, I_2$ and $I_3$ step by step. 
		
		{\bf Estimate on $I_1.$ } 
		It is clear that $|x|+|y|^2\le 3 t_0$ when $x\in (t, t_0)$ and $y^2\le 2 t_0.$ Lemma \ref{RY} 
		helps us to write  
		\begin{equation}\label{total0}
			\begin{aligned}
				I_1&=\int_t^{t_0} \int_{y^2<2 t_0}\frac{\widetilde{K}_n(z, z)}{2(\gamma_n n)^{3 / 4}} \mathrm{~d} y \mathrm{~d} x\\
				&=\frac{\left(1+O\left(\gamma_n^{-1}\right)\right)}{\sqrt{\pi}}(\frac{\log n}{2\gamma_n})^{5/4}\int_t^{t_0}\int_{y^2<2 t_0} e^{-x-y^2}\bigg(1-\frac{1+x+y^2+\frac12(x+y^2)^2}{\gamma_n}\bigg)\mathrm{~d} y \mathrm{~d} x.
			\end{aligned}
		\end{equation} 
		Using the well-known integrals on $$\int_t^{t_0} x^k e^{-x} dx \quad\text{and}\quad \int_{y^2<2 t_0}y^{2k} e^{-y^2} dy$$ for $k=0, 1, 2,$ we have  
		$$
		\begin{aligned}
			M:&=\frac{1}{\sqrt{\pi}}\int_t^{t_0}\int_{y^2<2 t_0}e^{-x-y^2} \bigg(1-\frac{1+x+y^2+\frac12(x+y^2)^2}{\gamma_n}\bigg)\mathrm{~d} y \mathrm{~d} x\\
			&=\left[(1-\gamma_n^{-1})(e^{-t}-e^{-t_0})-\gamma_n^{-1}\left(e^{-t}(\frac{t^2}2+2t+2)-e^{-t_0}(\frac{t_0^2}2+2t_0+2)\right)\right]{\rm erf}(\sqrt{2t_0})\\
			&\quad-\gamma_n^{-1}(e^{-t}-e^{-t_0})\left[\frac{7}{8}{\rm erf}(\sqrt{2t_0})-e^{-2 t_0}\big(\frac{(2t_0)^{3/2}}{2\sqrt{\pi}}+\frac{7(2t_0)^{1/2}}{4\sqrt{\pi}}\big)\right]\\
			&\quad-\gamma_n^{-1}((t+1)e^{-t}-(t_0+1)e^{-t_0})\left[\frac{1}{2}{\rm erf}(\sqrt{2t_0})-\frac{\sqrt{2t_0}}{\sqrt{\pi}}e^{-2 t_0}\right]
		\end{aligned}
		$$  and then it holds with simple calculus that 
		$$\aligned M=& {\rm erf}(\sqrt{2t_0})\left[e^{-t}\big(1-\gamma_n^{-1} (\frac{t^2}{2}+\frac{5t}{2}+\frac{35}8)\big)-e^{-t_0}\big(1-\gamma_n^{-1} (\frac{t_0^2}{2}+\frac{5t_0}{2}+\frac{35}8)\big)\right]\\
		&\quad\quad\quad\quad\quad\quad+\frac{\sqrt{2t_0}}{2}\gamma_n^{-1}e^{-2t_0}\left[e^{-t}(2t_0+\frac{11}2+2t)-e^{-t_0}(4t_0+\frac{11}2)\right]. \endaligned $$ 
		Thus, using the relationship ${\rm erf}(t)+{\rm erfc}(t)\equiv 1$ and \eqref{erfc}, we know from $t_0=4(|t|+(\log n)^{1/4})$ that
		$${\rm erf}(\sqrt{2t_0})=1+O(t_0^{-1/2}e^{-2 t_0})=1+o((\log n)^{-8}).$$
		Picking up the dominated term in the expression of $M,$ we finally get 
		\begin{equation}\label{total}
			\begin{aligned}
				M&=e^{-t}\big(1-\gamma_n^{-1} (\frac{t^2}{2}+\frac{5t}{2}+\frac{35}8)\big)(1+o((\log n)^{-8})).
			\end{aligned}
		\end{equation}	
		Substituting \eqref{total} into \eqref{total0}, one obtains 
		$$
		I_1=(\frac{\log n}{2\gamma_n})^{5/4}e^{-t}\big(1-\gamma_n^{-1} (\frac{t^2}{2}+\frac{5t}{2}+\frac{35}8)\big)(1+O((\log n)^{-1})).\\
		$$

		{\bf Estimates on $I_2$ and $I_3$. }
		Review the expression (30) in \cite{Cipolloni22Directional} stating that \begin{equation}\label{quote formula5}
			\frac{\widetilde{K}_n(z, z)}{(\gamma_n n)^{3/4}}\lesssim|z|^2 e^{-(x+y^2)/3}
		\end{equation}
		once  $x+y^2\geq 0,$ which is verified for both the regime $x\in (t, t_0)$ and $y^2\ge 2t_0$ and the other case $x\ge t_0$ and $y\in\mathbb{R}.$ Hence, it follows from \eqref{z2express} and \eqref{quote formula5} that $$
		\begin{aligned}
			|I_{2}|&=\int_t^{t_0} \int_{y^2 \geq 2 t_0}\frac{\widetilde{K}_n(z, z)}{2(\gamma_n n)^{3 / 4}} \mathrm{~d} y \mathrm{~d} x&\\
			&\lesssim \int_{t}^{t_0}\int_{y^{2}\geq 2t_{0}}  |z|^2 e^{-\frac{x+y^{2}}{3}}\mathrm{~d} y \mathrm{~d} x\\
			&=\int_{t}^{t_0}\int_{y^{2}\geq 2 t_0} \big(1+\sqrt{\frac{\gamma_n}{n}}+\frac{x+y^2}{\sqrt{n\gamma_n}}+\frac{(x+\gamma_n)^2}{4n\gamma_n} \big)  e^{-\frac{x+y^{2}}{3}}\mathrm{~d} y \mathrm{~d} x.
		\end{aligned}
		$$
		It is ready to see from the fact  $$|\int_t^{t_0} e^{-\frac x3} (1+c_1 x+c_2 x^2) dx|<+\infty$$ that 
		$$
		\begin{aligned}
			|I_2|&\lesssim \int_{y\geq \sqrt{2 t_0}}(2+\frac{y^2}{\sqrt{n\gamma_n}})e^{-\frac{y^{2}}{3}}\mathrm{~d} y\\
			& \lesssim {\rm erfc}\left(\sqrt{\frac{2t_0}{3}}\right) + \frac{\sqrt{2t_0}}{\sqrt{n\gamma_n}} e^{-\frac{2t_0}{3}}\\
			&=O((t_0)^{-1/2} e^{-2t_0/3})\\
			&\ll e^{-t} (\log n)^{-1}.
		\end{aligned}
		$$
		Here, the last inequality is due to the choice $t_0=4(|t|+(\log n)^{1/4}).$
		The upper bound \eqref{quote formula5} and similar argument as for $I_2$ work successfully for $I_3$ to ensure that $$I_3\ll e^{-t}(\log n)^{-1}$$ since 
		$$\int_{\mathbb{R}} e^{-\frac{y^2}{3}} (1+y^2)dy<+\infty \quad \text{and} \quad \int_{t_0}^{\infty} e^{-\frac x3}(1+c_1 x+c_2 x^2) dx\ll e^{-t}(\log n)^{-1}.$$ Putting estimates on $I_1, I_2$ and $I_3,$ we finally get 
		$$
		\operatorname{Tr} (W_n^{(t)})=(\frac{\log n}{2\gamma_n})^{5/4}e^{-t}\big(1-\gamma_n^{-1} (\frac{t^2}{2}+\frac{5t}{2}+\frac{35}8)\big)\left(1+O\left((\log n)^{-1}\right)\right).
		$$ Review $$2\gamma_n=\log n -5\log \log n -\log (2\pi^4).$$ Thus, 
		$$\aligned 
		&\quad(\frac{\log n}{2\gamma_n})^{5/4}\big(1-\gamma_n^{-1} (\frac{t^2}{2}+\frac{5t}{2}+\frac{35}8)\big)\\
		&=(1+\frac{25\log\log n+5\log (2\pi^4)}{4\log n})\big(1-\gamma_n^{-1} (\frac{t^2}{2}+\frac{5t}{2}+\frac{35}8)\big)\left(1+O\left((\log n)^{-1}\right)\right)\\
		&=(1+\frac{25\log\log n+5\log (2\pi^4)-4t^2-20t-35}{4\log n})\left(1+O\left((\log n)^{-1}\right)\right)\\
		&=(1+\frac{c_n-t^2-5t}{\log n})\left(1+O\left((\log n)^{-1}\right)\right),
		\endaligned $$
		where $c_n=(25\log \log n+5\log(2\pi^4)-35)/4.$
		The proof is then completed. 
	\end{proof}

\subsubsection{Estimates related to the spectral radius} 
	
		Recall 	$$
	\mathbb{P}(Y_n\le t)=\mathbb{P}\big(\max_{i} |\sigma_i|\le 1+\sqrt{\frac{\gamma_n^{\prime}}{4n}}+\frac{t}{\sqrt{4\gamma_n^{\prime} n}}\big), $$ where $\gamma_n^{\prime}=\log{n}-2\log{\log{n}}-\log{2\pi} $  and  then we set \begin{equation}\label{A_t spectral radius}
		{A}'(t):=\left\{z \in \mathbb{C} \left\lvert\, |z| \geq 1+\sqrt{\frac{\gamma_n^{\prime}}{4 n}}+\frac{t}{\sqrt{4 \gamma_n^{\prime} n}}\right.\right\}
	\end{equation}
	and $$\overline{W}_n^{(t)}:=\sqrt{\chi_{A'(t)}^{}} \widetilde{K_n}\sqrt{\chi_{A'(t)}^{}}.$$ 
	The same argument implies that 
	$$\mathbb{P}(Y_n\le t)={\rm det}(1-\overline{W}_n^{(t)}).$$ 
Now, we are going to obtain precise asymptotic for 	$\operatorname{Tr} (\overline{W}_n^{(t)})$ and similar upper bound as \eqref{quote formula4} for $\|\overline{W}_n^{(t)}\|_2.$  
	
	\begin{lem} \label{lemma2.4}
		Let $\overline{W}_n^{(t)}$ be defined as above with $|t| \le \sqrt{\log n}$ and set $d_n=2\log\log n+\log(2\pi)$. Then, 
		\begin{equation}\label{Tr or complex ginibre}
			\operatorname{Tr} (\overline{W}_n^{(t)})=				e^{-t}\bigg(1+\frac{d_n-t^2-4t}{\gamma_n'}\bigg)\left(1+O\left((\log n)^{-1}\right)\right)
		\end{equation}
		uniformly on $|t|\lesssim (\log n)^{1/4}$ and \begin{equation}\label{norm of complex ginibre}
			\|\overline{W}_n^{(t)} \|_2\lesssim e^{-\sqrt{\log n}}
		\end{equation}
		uniformly in $n$ and in $|t|\le  \sqrt{\log n}.$
	\end{lem}
	\begin{proof}  By \eqref{sumexpress} and \cite[Lemma 3.2]{Rider14}, we get
	via spherical coordinate that 	
		\begin{align}
			\operatorname{Tr}(\overline{W}_n^{(t)} )&=\int_{A'(t)} \widetilde{K_n}(z,z) \mathrm{d}^2 z \nonumber\\
			&=\int_{|z|\geq 1+\sqrt{\frac{\gamma_n^{\prime}}{4n}}+\frac{t}{\sqrt{{4n\gamma_n^{\prime}}}}}  \frac{n}{\pi} \frac{\Gamma(n,n|z|^2)}{\Gamma(n)}\mathrm{d}^2 z \nonumber\\
			&=2 n\int_{1+\sqrt{\frac{\gamma_n^{\prime}}{4n}}+\frac{t}{\sqrt{{4n\gamma_n^{\prime}}}}}^{+\infty} \frac{r \Gamma(n, n \,r^2)}{\Gamma(n)}\mathrm{d} r. \label{Trspec}
			 \end{align}
		
Using the substitution $x=r^2$ and rewriting the RHS of \eqref{Trspec} as a double integral  and then changing the order of integration, we have 
$$\operatorname{Tr}(\overline{W}_n^{(t)} )=\Gamma^{-1}(n) \int_{n a}^{+\infty}x^{n-1}e^{-x}(x-na) dx=\Gamma^{-1}(n)\big(\Gamma(n+1, na)-na\Gamma(n, na)\big)$$ and then we get in further via the relationship $$\Gamma(n+1, u)=n\Gamma(n, u)+u^n e^{-u}$$	that 	
\begin{equation}\label{Trsum}\operatorname{Tr}(\overline{W}_n^{(t)} )=\Gamma^{-1}(n)\big(-n(a-1)\Gamma(n, na)+(na)^{n}e^{-na}\big).\end{equation}		
Here, we set for simplicity $$a_t:=(1+\sqrt{\frac{\gamma_n^{\prime}}{4n}}+\frac{t}{\sqrt{{4n\gamma_n^{\prime}}}})^2=1+\sqrt{\frac{\gamma_n^{\prime}}{n}}+\frac{t}{\sqrt{n\gamma_n' }}+\frac{(\gamma_n'+t)^2}{4n\gamma_n'}$$ and then 
\begin{equation}\label{amu}\aligned n\mu^2(a_t)=n (a_t-1-\log a_t)=\frac{\gamma_n'}{2}+\frac{t^2}{\gamma_n'}+t+O(n^{-1/2} \gamma_n'^{3/2}).\endaligned \end{equation}
Using the same analysis as \eqref{2.14}-\eqref{2.12}, we get 
\begin{equation}\label{Trsum1}\frac{n(a_t-1)\Gamma(n, na_t)}{\Gamma(n)}=\sqrt{\frac{n}{2\pi}}e^{-n\mu^2(a_t)}(1-\frac{1}{2n\mu^2(a_t)}+O(n^{-2}\mu^{-4}(a_t))).\end{equation}
Also, we have by Stirling's formula and the fact $a_t-1=o(1)$ that 
\begin{equation}\label{Trsum2}\aligned\Gamma^{-1}(n)(na_t)^n e^{-na_t}&=\frac{\sqrt{n}}{\sqrt{2\pi}}e^{-n(a_t-1)+n\log a_t}(1+O(n^{-1}))=\sqrt{\frac{n}{2\pi}}e^{-n\mu^2(a_t)}(1+O(n^{-1})).\endaligned\end{equation}		
Putting \eqref{Trsum1} and \eqref{Trsum2} back into \eqref{Trsum}, thinking of \eqref{amu} and applying the Taylor formula when necessary, we derive that  \begin{equation}\label{Trsum0} \aligned \operatorname{Tr}(\overline{W}_n^{(t)} )&=\frac{1+O(\gamma_n'^{-1})}{2n\mu^2(a_t)}\sqrt{\frac{n}{2\pi}}e^{-n\mu^2(a_t)}\\
&=\sqrt{\frac{n}{2\pi}}\exp(-t-\frac{\gamma_n'}{2}-\frac{t^2}{\gamma_n'}) \frac{1+O((\gamma_n' )^{-1})}{\gamma_n'(1+\frac{4t}{\gamma_n'})}\\
&=\gamma_n'^{-1}\sqrt{\frac{n}{2\pi}}\exp(-t-\frac{\gamma_n'}{2})(1-\frac{t^2+4t}{\gamma_n'}))(1+O(\gamma_n'^{-1})). 
\endaligned 
\end{equation} 
Now  
$\gamma_n^{\prime}=\log{n}-2\log{\log{n}}-\log{2\pi}=\log n-d_n$ ensures that 
$$\operatorname{Tr}(\overline{W}_n^{(t)} )=e^{-t}\frac{\log n}{\gamma_n'}(1-\frac{t^2+4t}{\gamma_n'}))(1+O(\gamma_n'^{-1}))=e^{-t}(1+\frac{d_n-(t^2+4t)}{\gamma_n'})(1+O((\log n)^{-1})).$$
		
		Now it remains to show the upper bound \eqref{norm of complex ginibre}.  We cut $A^2(t)$ into two parts $A_1^2(t)$ and $A_2^2(t),$ respectively,  according to the events 
		$$\big\{|z|+|w|\geq 2 \sqrt{a_{\epsilon_n}}\big\} \quad \text{and} \quad \big\{|z|+|w|\leq 2 \sqrt{a_{\epsilon_n}} 
					\big\}$$ with $t\le \epsilon_n\ll \log n$ to be determined later.  
					Due to the elementary property $$|\widetilde{K}_n(z, w)|^2\le \widetilde{K}_n(z, z) \widetilde{K}_n(w, w)$$ for any $z, w\in\mathbb{C},$ it holds that  \begin{equation}\label{w2norm}
			\begin{aligned}
				\|\overline{W}_n^{(t)} \|_2^2
				&=\iint_{A^2(t)}|\widetilde{K}_n(z,w)|^2\mathrm{d}^2z\mathrm{d}^2w\\
				&=\iint_{A_1^2(t)}|\widetilde{K}_n(z,w)|^2\mathrm{d}^2z\mathrm{d}^2w+\iint_{A_2^2(t)}|\widetilde{K}_n(z,w)|^2 \mathrm{d}^2z\mathrm{d}^2w\\
				&=:J_1+J_2.
			\end{aligned}
		\end{equation}
		Since $$\big\{|z|+|w|\geq 2\sqrt{a_{\epsilon_n}}\big\}\subset \big\{|z|, |w|\geq \sqrt{a_{\epsilon_n}^{}}\big\},$$ 
		we get  
		$$\aligned J_1\le \big(\int_{|z|\ge \sqrt{a_{\epsilon_n}^{}}} \tilde{K}_n(z, z)d^2 z\big)^2
		\endaligned $$
and then similarly as \eqref{Trsum0} for the ${\rm Tr}(\overline{W}_n^{(t)}),$ we get 		
		\begin{equation}\label{J1}
			J_1\lesssim \frac{1}{n\mu^4(a_{\epsilon_n}^{})} e^{-2 n\mu^2(a_{\epsilon_n})}\lesssim\frac{n}{(\log n)^2} e^{-\gamma_n'-2\epsilon_n}=e^{-2\epsilon_n}.
		\end{equation}
For the second integral at the RHS of \eqref{w2norm}, it follows from the definition that 
$$|\widetilde{K}_n(z, w)|^2=\frac{n^2}{\pi^2}\exp(-n(|z|^2+|w|^2-z\bar{w}-\bar{z}w))\frac{\Gamma(n, n z\bar{w})\Gamma(n, n\bar{z}w)}{\Gamma^2(n)}$$	
and we	
		rescale the variables $z, w$ as $$
	z=r_1 e^{i\theta_1} \qquad \text{and} \qquad 	w=r_2e^{i\theta_2} 
	$$
	with $$r_1, \; r_2\ge \sqrt{a_t} \quad \text{and} \quad r_1+r_2\le 2\sqrt{a_{\epsilon_n}}.$$	
	When $|\theta_1-\theta_2|\ge  \sqrt{\frac{\log n}{n}},$ we invoke the asymptotics \cite[Lemma 3.4]{Rider14} $$
\frac{\Gamma(n,nz\bar{w})}{\Gamma (n)}=e^{-nz\bar{w}}\frac{(ez\bar{w})^n}{\sqrt{2\pi n}(1-z\bar{w})}\bigg(1+O\bigg(\frac{1}{n|1-z\bar{w}|^2}\bigg)\bigg)
$$
with $$|1-z\bar{w}|^2=1+r_1^2r_2^2-2r_1r_2\cos(\theta_1-\theta_2).
$$ 
Thus, for such $z$ and $w$
$$|\widetilde{K}_n(z, w)|^2\lesssim\frac{n}{1+r_1^2r_2^2-2r_1r_2\cos(\theta_1-\theta_2)}\exp(-n(r_1^2+r_2^2))e^{2n}r_1^{2n}r_2^{2n}.$$	
Standard integral tells $$\iint_{[0, 2\pi]^2} \frac{1}{1+b^2-2b\cos(\theta_1-\theta_2)} d\theta_1 d\theta_2=\frac{4\pi^2}{b^2-1}\lesssim \frac{1}{b-1}$$ for $1<b.$  
Thereby, we have 
$$\aligned &\quad\iint_{A^2_2(t), |\theta_1-\theta_2|\ge \sqrt{\frac{\log n}{n} }}|\widetilde{K}_n(z,w)|^2 \mathrm{d}^2z\mathrm{d}^2w \\
&\lesssim  n e^{2n}\iint_{r_1, r_2\ge \sqrt{a_t}, \,r_1+r_2\le 2\sqrt{a_{\epsilon_n}^{}}}\frac{1}{r_1r_2-1}\exp(-n(r_1^2+r_2^2))r_1^{2n}r_2^{2n}dr_1 dr_2. 
\endaligned $$ 
The maximum of the integrant is attained at $r_1=r_2=\sqrt{a_t}$ and the area of domain of integration $\{r_1, r_2\ge \sqrt{a_t}, \,r_1+r_2\le 2\sqrt{a_{\epsilon_n}^{}}\}$ is $2(a_{\epsilon_n^{}}-a_t)^2$ and therefore 
\begin{equation}\label{J21}\aligned \iint_{A^2_2(t), |\theta_1-\theta_2|\ge \sqrt{\frac{\log n}{n} }}|\widetilde{K}_n(z,w)|^2 \mathrm{d}^2z\mathrm{d}^2w \lesssim  \frac{n}{a_t-1}  e^{-2n \mu^2(a_t)}(a_{\epsilon_n^{}}-a_t)^2. 
\endaligned \end{equation} 

While for the case $|\theta_1-\theta_2|\le \sqrt{\frac{\log n}{n}},$ we use the asymptotics formula (47) in \cite{Cipolloni22Directional} listed as follows $$
	\frac{\Gamma(n,nz\bar{w})}{\Gamma(n)}=\frac{e^{-\frac{n(z\bar{w}-1)^2}{2}}}{\sqrt{2\pi n}(z\bar{w}-1)}\bigg(1+O\bigg(|z\bar{w}-1|+n|z\bar{w}-1|^2+\frac{1}{n|z\bar{w}-1|^2}\bigg)\bigg)
	$$
	and here, 
	$$\aligned |1-z\bar{w}|^2&= 1+r_1^2r_2^2-2r_1r_2(1-\frac{(\theta_1-\theta_2)^2}{2}+O(n^{-2}(\log n)^2))\\
	&=1+r_1^2r_2^2-2r_1r_2+r_1r_2(\theta_1-\theta_2)^2+O(n^{-2}(\log n)^2)\\
	&=O( \frac{\log n}{n}).
\endaligned$$
Thereby, it follows that $|\widetilde{K}_n(z, w)|^2$ has the same order as
\begin{equation}\label{integrant}\frac{n^2}{\log n}\exp\left(-n(r_1^2+r_2^2+1)+4nr_1r_2\cos(\theta_1-\theta_2)-nr_1^2r_2^2 \cos(2(\theta_1-\theta_2))\right).
\end{equation}
Under the condition $b=r_1 r_2=1+O(\sqrt{\log n/n})$ and $|\theta_1-\theta_2|\le \sqrt{\log n/n},$ one gets that  
$$4nb\cos(\theta_1-\theta_2)-nb^2 \cos(2(\theta_1-\theta_2))=4nb- n b^2+o(1)$$
and then the term in \eqref{integrant} has the same order as 
$$\frac{n^2}{\log n}\exp\left(-n(r_1^2+r_2^2+1)+4nr_1r_2-nr_1^2r_2^2 \right). $$
This helps us to derive \begin{equation}\label{J22}
		\begin{aligned}
				&\quad\iint_{A^2_2(t), |\theta_1-\theta_2|\le \sqrt{\frac{\log n}{n} }}|\widetilde{K}_n(z,w)|^2 \mathrm{d}^2z\mathrm{d}^2w \\
&\lesssim  \frac{n^2}{\log n} \sqrt{\frac{\log n}{n}} \iint_{r_1, r_2\ge \sqrt{a_t}, \,r_1+r_2\le 2\sqrt{a_{\epsilon_n}^{}}}\exp(-n(r_1^2+r_2^2+1)+4n r_1r_2-nr_1^2 r_2^2) dr_1 dr_2\\
&\lesssim \frac{n^{3/2}}{\sqrt{\log n}} e^{-n (a_t-1)^2}(a_{\epsilon_n^{}}-a_t)^2.
		\end{aligned}
	\end{equation}
Here, the last inequality is due to the same reason as that for \eqref{J21}. Combining \eqref{J21} and \eqref{J22} together, and recalling the fact $2n \mu^2(a_t)=n(a_t-1)^2+o(1)$, we have that 
$$J_2\lesssim \frac{n^{3/2}}{\sqrt{\log n}} e^{-n (a_t-1)^2}(a_{\epsilon_n^{}}-a_t)^2=\frac{n^{3/2}}{\sqrt{\log n}} e^{-\gamma_n'-2t}\frac{(\epsilon_n-t)^2}{n\gamma_n'}=\sqrt{\frac{\log n}{n}} e^{-2t}(\epsilon_n-t)^2.$$ We choose $\epsilon_n=\sqrt{\log n}$  to make sure that 
\begin{equation}\label{J2f}
	J_2\lesssim \frac{(\log n)^{3/2}}{\sqrt{n}} e^{2 \sqrt{\log n}}.  
\end{equation}	
Reviewing  $J_1\lesssim e^{-2\sqrt{\log n}},$
 we deduce that \begin{equation}
	\|\overline{W}_n^{(t)} \|_2^2=J_1+J_2 \lesssim e^{-2\sqrt{\log n}}. 
	\end{equation} The proof is completed now. 
\end{proof}

\subsection{Real Ginibre case} In this subsection, we aim to present analogs of the complex Ginibre ensembles.  
Based on the results in \cite{Poplavskyi17, Rider14}, one knows that 
$$\max_{\sigma_i\in\mathbb{R}} \sigma_i=1+O(\frac{1}{\sqrt{n}})$$  and by contrast it follows from the limit \eqref{Xurightmostginibre} ( \cite{Cipolloni22Directional}) that   $$ \quad \max_{\sigma_i\in\mathbb{C}_+\setminus\mathbb{R}}|\sigma_i|=1+O(\frac{\sqrt{\gamma_n}}{\sqrt{n}}).$$ Thus, the largest real eigenvalue is negligible in comparison to the largest modulus of complex eigenvalues. Thus, we only need to focus on the correlation kernels of complex eigenvalues. We now recall \cite[(2.2)]{Rider14}, which is listed as follows:
$$
K_{n}^{\mathbb{C}, \mathbb{C}}(z, w) := \begin{pmatrix} S_{n}(z, w) & -\mathbf{i} S_{n}(z, \overline{w}) \\ -\mathbf{i} S_{n}(\overline{z}, w) & S_{n}(w, z) \end{pmatrix}$$
with
$$\begin{aligned}
	S_{n}(z, w) &:= \frac{\mathbf{i} n e^{-n(z - \overline{w})^2 / 2}}{\sqrt{2 \pi}} \sqrt{n} (\overline{w} - z) \sqrt{\operatorname{erfc}(\sqrt{2 n} |\Im z|) \operatorname{erfc}(\sqrt{2 n} |\Im w|)} e^{-n z \overline{w}} K_{n}(z, w)\\
	&= \Phi_{n}(z, w) \widetilde{K}_{n}(z, w)
\end{aligned}$$and
$$\Phi_{n}(z, w) :=\mathbf{i}\sqrt{\frac{n\pi}{2}} (\overline{w} - z) e^{\frac12n(|z|^2 + |w|^2 - 2 z \overline{w})} e^{-\frac12n(z - \overline{w})^2 }\sqrt{\operatorname{erfc}(\sqrt{2 n} |\Im z|) \operatorname{erfc}(\sqrt{2 n} |\Im w|)}.$$

The following results come from \cite{Cipolloni22Directional} and \cite{Rider14}. For test functions $g: \mathbb{C} \to [0, 1]$ satisfying $g(\overline{z}) = g(z)$ and $g(x) = 0$, $\forall \; x \in \mathbb{R}$, it holds that
$$\mathbb{E} \prod_{i=1}^{n} (1 - g(\sigma_i)) = \left[ \det(1 - \sqrt{g} K_{n}^{\mathbb{C}, \mathbb{C}} \sqrt{g}) \right]^{1/2}.$$ 
For the rightmost complex eigenvalue, 
we set 
	$$
	W^{\mathbb{C},\mathbb{C},(t)}_n:= \sqrt{\chi_{A(t)}^{}}{K}_n^{\mathbb{C},\mathbb{C}}\sqrt{\chi_{A(t)}^{}}
	$$
	 with $A(t)=\left\{z \in \mathbb{C} \setminus\mathbb{R}\left\lvert\, \Re z \geq 1+\sqrt{\frac{\gamma_n}{4 n}}+\frac{t}{\sqrt{4 \gamma_n n}}\right.\right\}$
	and  
	we first provide a precise asymptotic for ${\rm Tr}(W^{\mathbb{C},\mathbb{C},(t)}_n),$ which turns out to be the same as that in Lemma \ref{Trlem}.
	\begin{lem}\label{Trlemr}
		Let   $W^{\mathbb{C},\mathbb{C},(t)}_n$ be defined as above with $|t| \lesssim (\log n)^{1/4}$ and $$c_n=(25\log \log n+5\log(2\pi^4)-35)/4.$$ Then, 
		\begin{equation}\label{2.14c}
			\operatorname{Tr} (W^{\mathbb{C},\mathbb{C},(t)}_n)=e^{-t}\big(1+\frac{c_n-t^2-5t}{\log n})\left(1+O\left((\log n)^{-1}\right)\right) 
		\end{equation}
		uniformly on $|t|\lesssim (\log n)^{1/4}$  and 
		$$\|W_n^{\mathbb{C}, \mathbb{C}, (t)}\|_2\lesssim e^{-\sqrt{\log n} / 32}$$ 
		uniformly on $|t|\le \sqrt{\log n}/10.$  
	\end{lem}
	\begin{proof}  The proof for the upper bound on $\|W_n^{\mathbb{C}, \mathbb{C}, (t)}\|$ was verified in \cite{Cipolloni22Directional}, which is the same as that for complex case. We only need to prove the precise expression \eqref{2.14c}. 
		By definition, we have 
		$$
		\begin{aligned}
			\operatorname{Tr} (W^{\mathbb{C},\mathbb{C},(t)}_n) & =2\int_{A(t)_+} \Phi_n(z, z)\widetilde{K}_n(z, z) \mathrm{d}^2 z . 
		\end{aligned}
		$$
		Here, $A(t)_+=A(t)\cap \{z\in\mathbb{c}|\Im z>0\}.$
		Now (54) in \cite{Cipolloni22Directional} states 
		$$\Phi_n(z, z)=\sqrt{2n\pi} {\Im \, z} e^{2n ({\Im}\, z)^2 }{\rm erfc}(\sqrt{2n}|{\Im}\, z|),$$ which 
		combining \eqref{erfc} implies  
		$$\Phi_n(z, z)=1+O(n^{-1/6}) \quad \text{if} \quad  \Im \, z\geq n^{-\frac{5}{12}}$$ and $\Phi_n(z, z)\le 1$ when ${\Im}\, z< n^{-\frac{5}{12}}.$ 
		Thus, it follows that 
		$$
			\begin{aligned}
				&\quad \operatorname{Tr} (W^{\mathbb{C},\mathbb{C},(t)}_n) \\
				&\le 2(1+O(n^{-1/6}))\int_{A(t)_+\cap \Im z\ge n^{-\frac{5}{12}}} \widetilde{K}_n(z, z)\mathrm{d}^2 z+2\int_{A(t)_+\cap \Im z< n^{-\frac{5}{12}}} \widetilde{K}_n(z, z)\mathrm{d}^2 z
			\end{aligned}
		$$
		and by contrast
		$$\begin{aligned}
				\operatorname{Tr} (W^{\mathbb{C},\mathbb{C},(t)}_n) 
				&\ge 2(1+O(n^{-1/6}))\int_{A(t)_+\cap \Im z\ge n^{-\frac{5}{12}}} \widetilde{K}_n(z, z)\mathrm{d}^2 z.
			\end{aligned}$$
Examining the proof of Lemma \ref{Trlem}, we know 
		$$\aligned2\int_{A(t)_+} \widetilde{K}_n(z, z)\mathrm{d}^2 z&=\left(1+O\left((\log n)^{-1}\right)\right)\int_{A(t)}\widetilde{K}_n(z, z)\mathrm{d}^2 z\\
		&=e^{-t}\big(1+\frac{c_n-t^2-5t}{\log n})\left(1+O\left((\log n)^{-1}\right)\right), 
		\endaligned$$
which is the claimed asymptotic for $\operatorname{Tr} (W^{\mathbb{C},\mathbb{C},(t)}_n).$ 

Thereby, the only thing we need to verify is 
\begin{equation}\label{remainsforTrcc} 
		J_3:=\int_{A(t)_+\cap \Im z< n^{-\frac{5}{12}}} \widetilde{K}_n(z, z)\mathrm{d}^2 z\lesssim e^{-t} (\log n)^{-1}.
		\end{equation}
Indeed, 
$$J_3=\int_{t}^{+\infty} \int_0^{\gamma_n^{1/4} n^{-1/6}}\frac{\widetilde{K}_n(z, z)}{2(\gamma_n n)^{3/4}}dy dx.$$	
We set $t_0=4(|t|+(\log n)^{1/4})$ again. 
When $t\ge t_0,$ and remembering $0<y<\gamma_n^{1/4} n^{-1/6}=o(1),$ we get from \eqref{quote formula5} and elementary inequality $2ab\le a^2+b^2$ that 
$$\frac{\widetilde{K}_n(z, z)}{2(\gamma_n n)^{3/4}}\le \big(1+\sqrt{\frac{\gamma_n}{n}}+\frac{x+y^2}{\sqrt{n\gamma_n}}+\frac{(x+\gamma_n)^2}{4n\gamma_n} \big)  e^{-\frac{x+y^{2}}{3}}\lesssim \big(1+\frac{x^2}{n\gamma_n} \big)  e^{-\frac{x}{3}}.$$
Hence 
\begin{equation}\label{trcc1}\aligned \int_{t_0}^{+\infty} \int_0^{\gamma_n^{1/4} n^{-1/6}}\frac{\widetilde{K}_n(z, z)}{2(\gamma_n n)^{3/4}}dy dx&\lesssim  \gamma_n^{1/4} n^{-1/6}\int_{t_0}^{+\infty}\big(1+\frac{x^2}{n\gamma_n} \big)  e^{-\frac{x}{3}} dx\\
&\lesssim \gamma_n^{1/4} n^{-1/6}e^{-t_0/3}.
\endaligned \end{equation}		
The regime $x\in (t, t_0)$ and $0<y<\gamma_n^{1/4} n^{-1/6}=o(1)$ suits the condition of Lemma \ref{RY} and then 
$$\frac{\widetilde{K}_n(z, z)}{2(\gamma_n n)^{3 / 4}}
		\lesssim e^{-x}, $$ which implies 
\begin{equation}\label{trcc2}\aligned &\quad  \int_t^{t_0}\int_0^{\gamma_n^{1/4} n^{-1/6}}\frac{\widetilde{K}_n(z, z)}{2(\gamma_n n)^{3/4}}dy dx&\le e^{-t} \gamma_n^{1/4} n^{-1/6}.
\endaligned \end{equation}
	
Combining the estimates \eqref{trcc1} and \eqref{trcc2}	together, we accomplish the verification of \eqref{remainsforTrcc}. The proof is then completed.	
			\end{proof}

Finally, we present a similar result concerning the largest modulus of complex eigenvalues in real Ginibre ensembles.
We define
 
	$$
	\overline{W}^{\mathbb{C},\mathbb{C},(t)}_n= \sqrt{\chi_{A(t)}^{}}{K}_n^{\mathbb{C},\mathbb{C}}(z,z)\sqrt{\chi_{A(t)}^{}},
	$$
where correspondingly  $$A(t):=\{z\in\mathbb{C}\setminus\mathbb{R}||z|\ge 1+\sqrt{\frac{\gamma_n'}{4n}}+\frac{t}{\sqrt{4n\gamma_n}}\}.$$ 

Following the proofs of Lemma \ref{Trlemr} and Lemma \ref{lemma2.4},  we obtain the following precise estimate for  	$\operatorname{Tr} (\overline{W}^{\mathbb{C},\mathbb{C},(t)}_n).$
	\begin{lem}\label{lemma2.6}
		Let  $\overline{W}^{\mathbb{C},\mathbb{C},(t)}_n$ be defined as above and $d_n$ be given in Lemma \ref{lemma2.4}. Then, 
		\begin{equation}\label{2.21}
			\operatorname{Tr} (\overline{W}^{\mathbb{C},\mathbb{C},(t)}_n)=				e^{-t}\bigg(1+\frac{d_n-(t^2+4t)}{\gamma_n'}\bigg)\left(1+O\left((\log n)^{-1}\right)\right)\end{equation} 
			uniformly on $|t|\lesssim (\log n)^{1/4}$ and  
			$$\|\overline{W}_n^{\mathbb{C}, \mathbb{C}, (t)}\|_2\le e^{-\sqrt{\log n}}$$
			uniformly on $|t|\le \sqrt{\log n}.$
		\end{lem}

	\section{Proofs of Theorems \ref{TH1} and \ref{TH2}}
	This section is devoted to the proofs of Theorem \ref{TH1} and \ref{TH2}.  We first give an important  proposition serving for both Berry-Esseen bound and $W_1$ Wasserstein distance.
	
\begin{prop}\label{Keyprop} Let ${\rm Tr}(W_n^{(t)}), {\rm Tr}(\overline{W}_n^{(t)})$ and ${\rm Tr}(W_n^{\mathbb{C}, \mathbb{C}, (t)}), {\rm Tr}(\overline{W}_n^{\mathbb{C}, \mathbb{C}, (t)})$ be defined as in the previous section. We have 
\begin{align} 
|e^{-{\rm Tr}(W_n^{(t)})}-e^{-e^{-t}}|&=e^{-e^{-t}-t}\frac{25\log\log n}{4\log n}(1+o(1)); \label{comright}\\
|e^{-{\rm Tr}(\overline{W}_n^{(t)})}-e^{-e^{-t}}|&=e^{-e^{-t}-t}\frac{2\log\log n}{\log n}(1+o(1)); \label{comsr}\\
|e^{-\frac12({\rm Tr}(W_n^{\mathbb{C}, \mathbb{C}, (t)}))}-e^{-\frac12e^{-t}}|&=e^{-\frac{1}{2}e^{-t}-t}\frac{25\log\log n}{8\log n}(1+o(1)); \label{realright}\\
|e^{-\frac12({\rm Tr}(\overline{W}_n^{\mathbb{C}, \mathbb{C}, (t)}))}-e^{-\frac12e^{-t}}|&=e^{-\frac12e^{-t}-t}\frac{\log\log n}{\log n}(1+o(1)) \label{realsr}
\end{align} 	
uniformly for $t\in[-\ell_1(n), \ell_2(n)]$ with $\ell_2(n)=(\log n)^{1/4}$ and $\ell_1(n)=\frac14\log \log n.$ 
\end{prop}

\begin{proof} We only prove \eqref{comright} and \eqref{realsr}. The others follow by same arguments. 
It is easy to see that  
\begin{equation}\label{trdiff} |e^{-{\rm Tr}(W_n^{(t)})}-e^{-e^{-t}}|=e^{-e^{-t}}|\exp(e^{-t}-{\rm Tr}(W_n^{(t)}))-1|.
\end{equation} 
Lemma \ref{Trlem} entails that 
$${\rm Tr}(W_n^{(t)})-e^{-t}=e^{-t}\frac{c_n-t^2-5t}{\log n}\left(1+O\left((\log n)^{-1}\right)\right)+e^{-t} O\left((\log n)^{-1}\right).$$	
The constraint $-\frac14\log\log n\le t\le (\log n)^{1/4}$ makes sure 
$${\rm Tr}(W_n^{(t)})-e^{-t}=O((\log n)^{3/4}/\log n)=o(1),$$ which is exactly the reason why we choose $\ell_1(n)=\frac14\log \log n$ and $\ell_2(n)=(\log n)^{1/4}.$ 
Hence, applying the Taylor's formula to the exponent of RHS of \eqref{trdiff} and using the definition of $c_n$, we derive 
$$\aligned|e^{-{\rm Tr}(W_n^{(t)})}-e^{-e^{-t}}|&=\frac{1+O((\log n)^{-1})}{\log n}e^{-e^{-t}-t}|c_n-t^2-5t+O(1)| \\
&=\frac{c_n}{\log n} e^{-e^{-t}-t}(1+o(1))\\
&=\frac{25\log\log n}{4\log n}e^{-e^{-t}-t}(1+o(1)).\endaligned $$
Next, we prove \eqref{realsr}. Similarly, Lemma \ref{lemma2.6} ensures that 
$$\aligned|e^{-\frac12({\rm Tr}(\overline{W}_n^{\mathbb{C}, \mathbb{C}, (t)}))}-e^{-\frac12e^{-t}}|&=e^{-\frac{1}{2}e^{-t}}|\exp(\frac{1}{2}e^{-t}-\frac{1}{2}{\rm Tr}(\overline{W}_n^{\mathbb{C}, \mathbb{C}, (t)}))-1|\\
&=\frac12e^{-\frac{1}{2}e^{-t}}|e^{-t}-{\rm Tr}(\overline{W}_n^{\mathbb{C}, \mathbb{C}, (t)})|(1+o(1))\\
&=\frac{\log\log n}{\log n}e^{-\frac{1}{2}e^{-t}-t} (1+o(1)).
\endaligned$$
\end{proof}

With this key proposition at hand, now we are going to work on the Berry-Esseen bound.  
			
	\subsection{The verification of Berry-Essen bound}  
	
	We first treat the rightmost eigenvalue of complex Ginibre and 
	we start with a similar decomposition as in \cite{MaMeng25}, which is 
	$$\aligned &\quad \quad\sup_{t \in \mathbb{R}}\left|\mathbb{P}(Z_{n}  \leq t)-e^{-e^{-t}} \right|\\
	&=
	\sup_{t \in \left[-\ell_1(n), \ell_2(n)\right]}\left|\mathbb{P}(Z_{n}   \leq t)-e^{-e^{-t}} \right|  +  \sup_{x\leq -\ell_1(n)}\left|\mathbb{P}(Z_{n}   \leq t)-e^{-e^{-t}} \right|
	+\sup_{x\geq \ell_2(n)}\left|\mathbb{P}(Z_{n}   \leq t)-e^{-e^{-t}} \right|. \endaligned 
	$$
  	The expression \eqref{comright}  guarantees that 
  	$$\aligned\sup_{-\ell_1(n)\le t\le \ell_2(n)}|e^{-{\rm Tr}(W_n^{(t)})}-e^{-e^{-t}}|&=\frac{25\log\log n(1+o(1))}{4\log n}\sup_{-\ell_1(n)\le t\le \ell_2(n)} e^{-t-e^{-t}}\\
  	&=\frac{25\log\log n(1+o(1))}{4e \log n}.\endaligned $$
  This, together with \eqref{Keydecom}, implies 
	\begin{equation}\label{commiddle}\sup_{t \in \left[-\ell_1(n), \ell_2(n)\right]}\left|\mathbb{P}(Z_{n}   \leq t)-e^{-e^{-t}} \right|=\frac{25\log\log n(1+o(1))}{4e \log n}. \end{equation} 
Now    
	$$
		\begin{aligned}
			\sup_{t\leq -\ell_1(n)}\left|\mathbb{P}(Z_{n}    \leq t)-e^{-e^{-t}} \right|&\leq \mathbb{P}(Z_{n}    \leq -\ell_1(n))+e^{-e^{\ell_1(n)}}\\
			\end{aligned}
	$$ 
	and then using \eqref{Keydecom} and Lemma \ref{Trlem}, 
	we have 
	\begin{equation}\label{estleft}
		\begin{aligned}		\sup_{t\leq -\ell_1(n)}\left|\mathbb{P}(Z_{n}    \leq t)-e^{-e^{-t}} \right|&\le \exp(-{\rm Tr}(W_n^{(-\ell_1(n))}))+2e^{-e^{\ell_1(n)}}+O(e^{-\sqrt{\log n}/(32)})\\
		&\lesssim e^{-e^{\ell_1(n)}}=e^{-(\log n)^{1/4}}.
		\end{aligned}
	\end{equation} 

	Similarly, we have 
	\begin{equation}\label{estright}\aligned
	\sup_{t\geq \ell_2(n)}\left|\mathbb{P}(Z_{n}   \leq t)-e^{-e^{-t}} \right|&\le \exp(-{\rm Tr}(W_n^{(\ell_2(n))}))+2(1-e^{-e^{-\ell_2(n)}})+O(e^{-\sqrt{\log n}/(32)})\\
	&\lesssim e^{-\ell_2(n)}=e^{-(\log n)^{1/4}}. \endaligned 
	\end{equation}	
	
	Combining (\ref{commiddle}), (\ref{estleft}) and (\ref{estright}) together, we derive the desired Berry-Esseen bound as 
	$$
	\sup_{t\in \mathbb{R}}\left|\mathbb{P}^{\rm Gin}(Z_{n}   \leq t)-e^{-e^{-t}} \right|= \frac{25\log \log n}{4e\log n}(1+o(1)).
	$$ 
	
	For the spectral radius of complex Ginibre ensembles, we only need to work on the middle part $\sup_{[-\ell_1(n), \ell_2(n)]}|\mathbb{P}(Y_n\le t)-e^{-e^{-t}}|,$ which contributes the precise asymptotic of the Berry-Esseen bound. 
	
	Indeed, \eqref{comsr} entails that 
	$$\aligned\sup_{t\in [-\ell_1(n), \ell_2(n)] }|\exp(-{\rm Tr}(\overline{W}_n^{(t)}))-e^{-e^{-t}}|&=\frac{2\log\log n(1+o(1))}{\log n}\sup_{t\in [-\ell_1(n), \ell_2(n)] } e^{-t-e^{-t}}\\
	&=\frac{2\log\log n}{ e\log n}(1+o(1)). 
	\endaligned $$
The upper bound  \eqref{norm of complex ginibre} in Lemma \ref{lemma2.4} brings similarly as \eqref{Keydecom} the following asymptotic 
$$|\mathbb{P}(Y_n\le t)-e^{-e^{-t}}|=|\exp(-{\rm Tr}(\overline{W}_n^{(t)}))-e^{-t}|+O(e^{-\sqrt{\log n}}).$$
This closes the proof of Berry-Esseen bound for $Y_n.$  	
The results for real Ginibre ensemble follow by the same argument and their proofs are omitted here. 	

	\subsection{On the $W_1$ distance}	
	We still verify the $W_1$ distance for rightmost eigenvalue of the complex Ginibre.  We decompose the integral into two parts as ,
	$$
	\begin{aligned}
		&\quad W_{1}\left(\mathcal{L}(F_n),\Lambda\right)\\&=\left(\int_{\ell_2(n)}^{\infty}+\int_{-\infty}^{-\ell_1(n)}\right)\left|\mathbb{P}(Z_{n}   \leq t)-e^{-e^{-t}}\right|\mathrm{~d} t + \int_{-\ell_1(n)}^{\ell_2(n)}\left|\mathbb{P}(Z_{n}   \leq t)-e^{-e^{-t}}\right|\mathrm{~d} t\\
		&=:I_1+I_2.
	\end{aligned}
	$$
	It suffices to prove 
	$$I_2=\frac{25\log\log n}{4\log n}(1+o(1)) \quad \text{and} \quad I_1\ll \frac{\log \log n}{\log n}.$$
	First, as for the Berry-Esseen bound,  we derive from \eqref{Keydecom} and \eqref{comright} that 
	\begin{equation*}\aligned 
		I_2
		&=\int_{-\ell_1(n)}^{\ell_2(n)} |e^{-{\rm Tr}(W_n^{(t)})}-e^{-e^{-t}}|\mathrm{~d} t+O(\ell_2(n)e^{-\frac1{32}\sqrt{\log n}})\\
		&=\frac{25\log\log n(1+o(1))}{4\log n} \int_{-\ell_1(n)}^{\ell_2(n)}  e^{-e^{-t}-t} dt+o(\frac{\log\log n}{\log n})\\
		&=\frac{25\log\log n}{4\log n}(1+o(1)),\endaligned
	\end{equation*}
	where the second equality is due to the  integral $$\int_{-\infty}^{+\infty} e^{-e^{-t}-t}  dt=1.$$ 	
	
	Now we investigate the behavior of $I_1.$ The Markov inequality helps us to get the following rough estimate
	$$
	\begin{aligned}
		\int_{|t|\ge \ell_2(n)}\left|\mathbb{P}(Z_{n}   \leq t)-e^{-e^{-t}}\right|\mathrm{~d} t &\leq 2\int^{\infty}_{\ell_2(n)}\mathbb{P}(|Z_{n}  | \geq t)+\mathbb{P}(|\Lambda|\geq t)\mathrm{~d} t \\
		&\lesssim \ell_2^{-4}(n)(\mathbb{E}|Z_{n}  |^5+\mathbb{E}|\Lambda|^5)\\
		&	\ll \frac{\log \log n}{\log n}.\\
	\end{aligned}
	$$
	It remains to prove 
	$$\int_{-\ell_2(n)}^{-\ell_1(n)}\left|\mathbb{P}(Z_{n}   \leq t)-e^{-e^{-t}}\right|\mathrm{~d} t\ll \frac{\log \log n}{\log n}.$$
Indeed, 
	\begin{equation}\aligned \label{E12}
		\int_{-\ell_2(n)}^{-\ell_1(n)} |e^{-{\rm Tr}(W_n^{(t)})}-e^{-e^{-t}}|dt\le \int_{-\ell_2(n)}^{-\ell_1(n)} e^{-e^{-t}} dt+\int_{-\ell_2(n)}^{-\ell_1(n)} e^{-{\rm Tr}(W_n^{(t)})} dt.
		\endaligned \end{equation}
	Lemma \ref{Trlem} and the fact $t^2+5t>0$ when $t\in (-\ell_2(n), -\ell_1(n))$ give us 
	$${\rm Tr}(W_n^{(t)})\ge e^{-t}(1+\frac{c_n}{\log n})(1+O((\log n)^{-1}))$$ and 
	the fact $c_n\gg 1$ helps us to get in further that 
	$$(1+\frac{c_n}{\log n})(1+O((\log n)^{-1}))=1+\frac{c_n}{\log n}+O((\log n)^{-1})\ge 1.$$ 
	Hence, 
	$${\rm Tr}(W_n^{(t)})\ge e^{-t}$$ and then, with \eqref{E12} and simple calculus, it holds 
	$$\int_{-\ell_2(n)}^{-\ell_1(n)} |e^{-{\rm Tr}(W_n^{(t)})}-e^{-e^{-t}}|dt\lesssim \int_{-\ell_2(n)}^{-\ell_1(n)} e^{-e^{-t}} dt\lesssim \ell_2(n) e^{-e^{\ell_1(n)}}\ll\frac{\log \log n}{\log n} $$ 
	since $\ell_1(n)=\frac14\log\log n$ and $\ell_2(n)=(\log n)^{1/4}.$
	
As far as the $W_1$ distance for the spectral radius of the complex ensemble, 	
	we similarly decompose the integral into two parts as follows
	$$
	\begin{aligned}
		&\quad W_{1}\left(\mathcal{L}(Y_n  ),\Lambda\right)\\&=\left(\int_{\ell_2(n)}^{\infty}+\int_{-\infty}^{-\ell_1(n)}\right)\left|\mathbb{P}(Y_{n}   \leq t)-e^{-e^{-t}}\right|\mathrm{~d} t + \int_{-\ell_1(n)}^{\ell_2(n)}\left|\mathbb{P}(Y_{n}  \leq t)-e^{-e^{-t}}\right|\mathrm{~d} t\\
		&=: I_1+I_2.
	\end{aligned}
	$$
		Repeating the proof for rightmost eigenvalue while using Lemma \ref{lemma2.4} instead we still have 
	$$I_2=\frac{2\log\log n}{\log n}(1+o(1)) \quad \text{and} \quad I_1\ll \frac{\log \log n}{\log n}$$
	which implies $$
	W_{1}\left(\mathcal{L}(Y_n  ),\Lambda\right)=\frac{2\log\log n}{\log n}(1+o(1)).
	$$
	The verification for the real Ginibre case follows from exactly the same argument, using Lemmas \ref{Trlemr} and \ref{lemma2.6}.
	\section{Proof of Theorem \ref{TH1} for IID case}
	In this section, we briefly demonstrate that the convergence rate of the rightmost eigenvalues established for the complex Ginibre ensemble also holds universally for complex i.i.d. matrices whose entries satisfy Assumption \ref{assumption}.
	
\begin{thm}
\label{TH3}
		Let $X$ be a complex    matrix with i.i.d. entries satisfying {\bf Assumption} \ref{assumption} and let $(\sigma_i)_{1\le i\le n}$ be its eigenvalues. Set $$Z_n=\sqrt{4 n \gamma_n}\left[\max_i\Re\sigma_i-1-\sqrt{\frac{\gamma_n}{4 n}}\right]. $$ 	
		Then, we have
		$$
		\sup_{x\in \mathbb{R}}|\mathbb{P}(Z_n\le x)-e^{-e^{-x}}|=\frac{25\log \log n}{4e\log n}(1+o(1))
		$$
		and	$$
		W_{1}\left(\mathcal{L}(Z_n),\Lambda\right)=\frac{25\log \log n }{4\log n}(1+o(1))
		$$
		for  sufficiently large $n$.	
\end{thm}
	
\begin{proof}
	We outline the sketch of proof  building upon both Theorem \ref{TH1} and the Green function comparison theorem introduced in \cite{Cipolloni23Universality}. 
	
	Let $F_n$ and $F_n^{\rm Gin}$ be distribution functions of $Z_n$ and $Z_n^{\rm Gin},$ respectively. Here the superscript {\rm Gin} represents the complex Ginibre case. 
	First of all, the triangle inequality ensures 
	$$ \aligned 
	|F_n(x)-e^{-e^{-x}}|&\le |F_n(x)-F_n^{\rm Gin}(x)|+|F_n^{\rm Gin}(x)-e^{-e^{-x}}|;	\\
	|F_n(x)-e^{-e^{-x}}|&\ge |F_n^{\rm Gin}(x)-e^{-e^{-x}}|-|F_n(x)-F_n^{\rm Gin}(x)|.
	\endaligned$$
	By the Berry-Esseen bound obtained for the complex Ginibre case, it suffices to prove 
	$$\sup_{x\in\mathbb{R}}|F_n(x)-F_n^{\rm Gin}(x)|\ll \frac{\log\log n}{\log n}. $$ 
	Similarly, we start with the decomposition  
	\begin{equation}\label{splitting}
		\sup_{x \in \left[-\ell_n, \ell_n\right]}\left|F_n(x)-F_n^{\rm Gin}(x)\right|  +  \sup_{|x| \geq \ell_n}\left|F_n(x)-F_n^{\rm Gin}(x)\right|=:I_1+I_2
	\end{equation}
	with $\ell_n=(\log n)^{1/4}.$

	Scrutinizing the proof of \cite[Theorem 4.1]{Cipolloni23Universality}, we have
	\begin{equation}\label{Berry-Esseen bound-middle}
		\sup_{x\in\left[-\ell_n,\ell_n\right]}\left|F_n(x)-F_n^{\mathrm{Gin}}(x)\right|\lesssim\frac{1}{n^{\epsilon}}
	\end{equation}
	for some $\epsilon>0,$ which constitutes the crux for translating the convergence rate from the i.i.d. scenario to the Ginibre ensemble. 
	
	Then, it  follows from  \eqref{Berry-Esseen bound-middle} that 
	\begin{equation}\label{Berry-Esseen bound-left}
		\begin{aligned}
			\sup_{x\le -\ell_n}\left|F_n(x)-F_n^{\mathrm{Gin}}(x)\right|&\le F_n(-\ell_n)+F_n^{\rm Gin}(-\ell_n) \\
			&\le \sup_{x\in\left[-\ell_n,\ell_n\right]}\left|F_n(x)-F_n^{\mathrm{Gin}}(x)\right|+2F_n^{\rm Gin}(-\ell_n)\\
			&\ll \frac{\log \log n }{\log n}.
		\end{aligned}
	\end{equation}
	Similarly, we have the following estimate 
	\begin{equation}\label{Berry-Esseen bound-right}
		\sup_{x\geq \ell_n}|F_n (x)-F_n^{\mathrm{Gin}}(x)|\ll \frac{\log\log n}{\log n}.
	\end{equation}
	Combining (\ref{Berry-Esseen bound-middle}), (\ref{Berry-Esseen bound-left}) and (\ref{Berry-Esseen bound-right}) together, we obtain $$
	\sup_{x\in \mathbb{R}} |F_n(x)-e^{-e^{-x}}|= \frac{25\log \log n}{4e\log n}(1+o(1)),
	$$
	which completes the proof of the Berry-Esseen bound in Theorem \ref{TH3}.
	
	Now we prove the result for $W_1$ Wasserstein distance. As for the Berry-Esseen bound, it suffices to show$$
	W_1(\mathcal{L}(Z_n),\mathcal{L}(Z_n^{\mathrm{Gin}}))\ll \frac{\log \log n}{\log n}.$$
	
	We start with  taking $\kappa_n=\left(\log{n}\right)^{1/6}$ which satisfies $\kappa_n \ll \ell_n $ and   decomposing the integral into two parts
	\begin{equation}\label{decomtwo}
		\begin{aligned}
			W_1(\mathcal{L}(Z_n),\mathcal{L}(Z_n^{\mathrm{Gin}}))&=\int_\mathbb{R}|F_n(x)-F_n^{\mathrm{Gin}}(x)| dx\\
			&=\left(\int_{|x|\leq \kappa_n}+\int_{|x|\geq \kappa_n} \right)\left|F_n(x)-F_n^{\mathrm{Gin}}(x) \right| dx.
		\end{aligned}
	\end{equation}

	Then the triangle and Markov inequalities help us to obtain
	\begin{equation}\label{second}
		\begin{aligned}
			\int_{|x|\geq \kappa_n}\left|F_n(x)-F_n^{\mathrm{Gin}}(x) \right| dx &\leq 2\int_{\kappa_n}^{+ \infty}\mathbb{P}(|Z_n|\geq x)+\mathbb{P}^{\mathrm{Gin}}(|Z_n | \geq x) dx \\
			&\leq  2\int_{\kappa_n}^{+ \infty}\frac{\mathbb{E}|Z_n|^{9}+\mathbb{E}^{\mathrm{Gin}}|Z_n |^{9}}{x^{9}} dx \\
			&\lesssim \frac{\mathbb{E}|Z_n|^{9}+\mathbb{E}^{\mathrm{Gin}}|Z_n |^{9}}{\kappa_n^{8}}\\
			&\ll\frac{\log{\log{n}}}{\log{n}},
		\end{aligned}
	\end{equation}
	where, the boundedness of the moments hold similarly  due to \cite[Remark 2.4]{Cipolloni23Universality}. 
	
	On the other hand, it follows again from \eqref{Berry-Esseen bound-middle} that  
	\begin{equation*}\aligned
		\int_{-\kappa_n}^{\kappa_n}|F_n(x)-F_n^{\mathrm{Gin}}(x)| dx \lesssim\frac{\kappa_n}{n^{\varepsilon}}\ll \frac{\log\log n}{\log n},
		\endaligned 
	\end{equation*} which with \eqref{second} confirms 
	\eqref{decomtwo}. 
	\end{proof} 
	
	\begin{rmk}
		Similarly, analogous results can be derived for the spectral radius of complex i.i.d. random matrices whose entries satisfy Assumption \ref{assumption}, which are  $$\sup_{t \in \mathbb{R}}\left|\mathbb{P}(Y_n \leq t)-e^{-e^{-t}} \right|=\frac{2\log \log n}{e\log n}(1+o(1)).
	$$
	and $$
	W_1(\mathcal{L}(Y_n), \Lambda)=\frac{2\log \log n}{\log n}(1+o(1)).
	$$
	\end{rmk}

\section{The scaling constants affect the convergence rates } 
In this last section, we briefly state the influence of scaling constants $\gamma_n$ and $\gamma_n'$ on the convergence rate and we summarize them as a theorem.  
\begin{thm} \label{TH4}
		Let $X$ be a  real or complex Ginibre ensemble  and let $\sigma_1,\cdots ,\sigma_n$  be the same as above.	Let 
		$\widetilde{\gamma}_n$ and $\widetilde{\gamma}_n'$ be the unique solution to the following equations, respectively 
		$$\aligned  64 \widetilde{\gamma}_n^5 \pi^4 \exp(2\widetilde{\gamma}_n)=n \quad \text{and} \quad 
				2\pi (\widetilde{\gamma}_n')^2\exp(\widetilde{\gamma}_n')=n.
		\endaligned $$ 
		Set 
		$$\aligned Z_n&=\sqrt{4 n \widetilde{\gamma}_n}\left[\max_i\Re\sigma_i-1-\sqrt{\frac{\widetilde{\gamma}_n}{4 n}}\right]; \\
		Y_n&=\sqrt{4 n \widetilde{\gamma}_n^{\prime}}\left[\max_i |\sigma_i|-1-\sqrt{\frac{\widetilde{\gamma}_n^{\prime}}{4 n}}\right]. \endaligned $$
		Then, for  sufficiently large $n$, we have		$$\aligned
		\sup_{x\in \mathbb{R}}|\mathbb{P}(Z_n\le t)-e^{-\frac{\beta}{2}e^{-t}}|&=\frac{\kappa_1}{4\log n}(1+o(1)); \\
		\sup_{t\in \mathbb{R}}|\mathbb{P}(Y_n\le t)-e^{-\frac{\beta}{2}e^{-t}}|&=\frac{\kappa_2}{\log n}(1+o(1)),
		\endaligned $$  
	where
		$$\aligned \kappa_1&=\sup_{t\in\mathbb{R}} \exp(-e^{-t}-t)(4t^2+20t+35)\approx 15.4; \\
		\kappa_2&=\sup_{t\in\mathbb{R}} \exp(-e^{-t}-t)(t^2+4t)\approx 1.48.
		\endaligned $$
	\end{thm}
\begin{proof}
The similar arguments still work and we only need to figure out the precise asymptotical expressions of ${\rm Tr}(W_n^{(t)})$ and ${\rm Tr}(\overline{W}_n^{(t)}).$ In fact, the expression \eqref{equa} is replaced by 
$$
\exp(-\frac{1}{2}\widetilde{\gamma}_n)=n^{-1/4}\widetilde{\gamma}_n^{5/4} \pi 2^{3/2},	
$$ which makes 
 \begin{equation}\label{equamodi}  \frac{\widetilde{K}_n(z, z)}{2(\widetilde{\gamma}_n n)^{3 / 4}}=\frac{ e^{-x-y^2}}{\sqrt{\pi}}\big(1-\widetilde{\gamma}_n^{-1}(1+x+y^2+\frac12(x+y^2)^2)\big)\left(1+O((\log n)^{-1})\right) \end{equation}
and then the expression \eqref{total} is modified to be  
$$\aligned 
		I_1=e^{-t}\big(1-\widetilde{\gamma}_n^{-1} (\frac{t^2}{2}+\frac{5t}{2}+\frac{35}8)\big)(1+O((\log n)^{-1})).\\
\endaligned $$
No change is observed in $I_2$ and $I_3$ and then 
$${\rm Tr}(W_n^{(t)})=e^{-t}\big(1-\widetilde{\gamma}_n^{-1} (\frac{t^2}{2}+\frac{5t}{2}+\frac{35}8)\big)(1+O((\log n)^{-1}))$$ uniformly on $|t|\lesssim (\log n)^{1/4},$
whence 
$$|e^{-{\rm Tr}(W_n^{(t)})}-e^{-e^{-t}}|=\widetilde{\gamma}_n^{-1}\exp(-e^{-t}-t)(\frac{t^2}{2}+\frac{5t}{2}+\frac{35}8)(1+o(1)).$$
Therefore, 
$$\sup_{t\in [-\ell_1(n), \ell_2(n)]}|e^{-{\rm Tr}(W_n^{(t)})}-e^{-e^{-t}}|=\widetilde{\gamma}_n^{-1} \sup_{t\in [-\ell_1(n), \ell_2(n)]}\exp(-e^{-t}-t)(\frac{t^2}{2}+\frac{5t}{2}+\frac{35}8).$$
The defining equation for $\widetilde{\gamma}_n$ ensures that $\widetilde{\gamma}_n=\frac{1}{2}\log n(1+o(1))$ and then 
$$\aligned \sup_{t\in \mathbb{R}}|\mathbb{P}(Z_n\le t)-e^{-e^{-t}}|&=\frac{1+o(1)}{4\log n}\sup_{t\in [-\ell_1(n), \ell_2(n)]}\exp(-e^{-t}-t)(4t^2+20t+35)\\
&= \frac{\kappa_1}{4\log n}(1+o(1)).\endaligned $$
Similarly, when $\widetilde{\gamma}_n'$ satisfies $2\pi (\widetilde{\gamma}_n')^2\exp(\widetilde{\gamma}_n')=n,$ the equation \eqref{Trsum0} is modified to be  
$$\aligned \operatorname{Tr}(\overline{W}_n^{(t)} )
&=\widetilde{\gamma}_n'^{-1}\sqrt{\frac{n}{2\pi}}\exp(-t-\frac{\widetilde{\gamma}_n'}{2})(1-\frac{t^2+4t}{\widetilde{\gamma}_n'}))(1+O(\widetilde{\gamma}_n'^{-1}))\\
&=e^{-t}(1-\frac{t^2+4t}{\widetilde{\gamma}_n'})(1+O(\widetilde{\gamma}_n'^{-1}))
\endaligned 
$$
and then 
$$|e^{-{\rm Tr}(\overline{W}_n^{(t)})}-e^{-e^{-t}}|=\frac{1}{\widetilde{\gamma}_n'}\exp(-e^{-t}-t)(t^2+4t)(1+o(1)).$$
The proof is then completed by taking the supremum and  the fact $\widetilde{\gamma}_n'=\log n(1+o(1)).$   
	\end{proof}


\begin{thebibliography}{SOSL90} 
	\bibitem{AB10} G, Akemann and M,  Bender(2010).  Interpolation between Airy and Poisson statistics for unitary chiral non-Hermitian random matrix ensembles. J. Math. Phys. {\bf 51}(10): 103524.
		\bibitem{Akemannand2012}	G. Akemann and  M. J. Phillips. Universality Conjecture for all Airy, Sine and Bessel Kernels in the Complex Plane. 	arXiv:1204.2740.
		
		\bibitem{AP14} G. Akemann and  M.J. Phillips. The Interpolating Airy Kernels for the $\beta=1$ and $\beta=4$ Elliptic Ginibre Ensembles. \emph{J. Stat. Phys.}, \textbf{155}(2014), 421-465. 
		
		\bibitem{Alt_2021}
		J. Alt, L. Erd\"{o}s and  T. Kr\"{u}ger. Spectral radius of random matrices with independent entries. \emph{Probab. Math. Phys.,} \textbf{2}(2)(2021), 221-280.
		
		\bibitem{Bai97} Z. D. Bai. Circular law. \emph{Ann. Probab.,} \textbf{25}(1)(1997), 494-529.
		
		\bibitem{Bai93i} Z. D. Bai. Convergence Rate of Expected Spectral Distributions of Large Random Matrices. Part I. Wigner Matrices. \emph{Ann. Probab.,} \textbf{21}(2)(1993), 625-648.
		
		\bibitem{Bai93ii} Z. D. Bai. Convergence Rate of Expected Spectral Distributions of Large Random Matrices. Part II. Sample Covariance Matrices. \emph{Ann. Probab.,} \textbf{21}(2)(1993), 649-672.
		
		\bibitem{Bai24} Z. D. Bai and J. Hu. A revisit of the circular law. arXiv:2408.13490.
		
		\bibitem{Bai88} Z. D. Bai and Y. Q. Yin. Convergence to the semicircle law. \emph{Ann. Probab.,} \textbf{16}(1988), 863-875.
		\bibitem{Bender10}
		M. Bender.  Edge scaling limits for a family of non-Hermitian random matrix ensembles. \emph{Probab. Th. Relat. Fields}, \textbf{147}(2010), 241-271.
		\bibitem{BChafai12} C. Bordenave and D. Chafa$\ddot{\rm \imath}$. Around the circular law. \emph{Probab. Surv.,} \textbf{9}(2012), 1-89. 
		
		\bibitem{Borodin09}
		A. Borodin and C. D. Sinclair. The Ginibre ensemble of real random matrices and its  scaling limits. \emph{Commun. Math. Phys.,} \textbf{291}(2009), 177-224.
		\bibitem{Borodin16} A. Borodin, M. Poplavskyi, C. D. Sinclair, R. Tribe, and O. Zaboronski, Erratum to: The Ginibre ensemble
		of real random matrices and its scaling limits [MR2530159], \emph{Comm. Math. Phys.,} \textbf{346}(2016), 1051-1055.
		\bibitem{ByunForrester}S. S. Byun and P. J. Forrester. \emph{Progress on the study of the Ginibre ensembles.} 1st ed., Springer Singapore, 2024.
		\bibitem{Cipolloni22Directional}\text{G. Cipolloni, L. Erd\"{o}s, D. Schr\"{o}der and Y. Xu.} Directional extremal statistics for Ginibre eigenvalues. \emph{J. Math. Phys.,} \textbf{63}(10)(2022), 103303.
		
		\bibitem{Cipolloni22rightmost}
		G. Cipolloni, L. Erd\"{o}s, D. Schr\"{o}der and Y. Xu. On the rightmost eigenvalue of non-Hermitian random matrices. \emph{Ann. Probab.,} \textbf{51}(6)(2022), 2192-2242.
		
		
		\bibitem{Cipolloni23Universality}
		G. Cipolloni, L. Erd\"{o}s and Y.  Xu. Universality of extremal eigenvalues of large random matrices. arXiv:2312.08325. 
		\bibitem{ChafaiP} D. Chafai and S. P\'ech\'e. A note on the second order universality at the edge of Coulomb gases on the plane. \emph{J. Stat.
			Phys.}, \textbf{156}(2014), 368-383.
		\bibitem{Edelman97}
		A. Edelman. The Probability that a random real Gaussian matrix has $k$ real eigenvalues, related distributions, and the circular law. \emph{J. Multivar. Anal.,} \textbf{60}(2)(1997), 203-232.
		\bibitem{ForresterNagao} P. Forrester and T. Nagao. Eigenvalue statistics of the real Ginibre ensemble. \emph{Phys. Rev. Lett.,} \textbf{99}(2007), 050603.
		\bibitem{Garrod15} B. Garrod, M. Poplavskyi, R. Tribe and O. Zaboronski. Examples of interacting particle systems on $\mathbb{Z}$ as Pfaffian point processes: annihilating and coalescing random walks. \emph{Ann. Henri. Poincar$\acute{e}$,} \textbf{19}(2018), 3635–3662. 
		\bibitem{Ginibre65} J. Ginibre. Statistical Ensembles of Complex, Quaternion, and Real Matrices. \emph{J. Math. Phys.,} \textbf{6}(1965), 440-449.
		
		\bibitem{Girko84} V. L. Girko. The circular law. \emph{Teor. Veroyatnost. i Primenen.,} \textbf{29}(1984), 669-679.
		
		\bibitem{Gohberg} I. Gohberg, S. Goldberg, and N. Krupnik. \emph{Traces and determinants of linear operators. Operator Theory: Advances and Applications.} {\bf 116} Birkhauser, Basel, 2000.
		\bibitem{Kostlan1992} E., Kostlan. On the spectra of Gaussian matrices. \emph{Linear Algebra Appl.} {\bf 162/164}(1992), 385-388.
		\bibitem{MaMeng25}  Y. T. Ma and X. Meng. Exact convergence rate of spectral radius of complex Ginibre to Gumbel distribution. arXiv:2501.08039.
		
		\bibitem{MaWang25}  Y. T. Ma and S. Wang. Optimal $W_1$ and Berry-Esseen bound between the spectral radius of large Chiral non-Hermitian random matrices and Gumbel. 	arXiv:2501.08661.
		
		\bibitem{Meckes15} E. Meckes and M. Meckes. A rate of convergence for the circular law for the complex Ginibre ensemble. \emph{Ann. Fac. Sci. Toulouse Math.,}\textbf{ 24}(2015), 93-117.
		
		\bibitem{Poplavskyi17} M. Poplavskyi, R. Tribe and O. Zaboronski. On the distribution of the largest real eigenvalue for the real Ginibre ensemble. \emph{Ann. Appl. Probab.,}\textbf{ 27}(3)(2017), 1395-1413.
		\bibitem{Rider03} B. Rider. A limit theorem at the edge of a non-Hermitian random matrix ensemble. \emph{J. Physics A.,}\textbf{36}(2003), 3401-3409.
		\bibitem{Rider14} B. Rider and C. D. Sinclair. Extremal laws for the real Ginibre ensemble. \emph{ Ann. Appl. Probab.}, \textbf{24}(4) (2014), 1621-1651.
		
		\bibitem{Bai95} J.W. Silverstein and Z. D. Bai. On the Empirical Distribution of Eigenvalues of a Class of Large Dimensional Random Matrices. \emph{J. Multivariate. Anal.,} \textbf{54}(6)(1995), 175-192.
		\bibitem{Sommers} H.-J. Sommers. Symplectic structure of the real Ginibre ensemble. \emph{J. Phys. A} \textbf{40}(2007), F671-F676, MR2371225. 
		
		\bibitem{TaoVu08} T. Tao and V. Vu. Random matrices: The circular law. \emph{ Commun. Contemp. Math.}, \textbf{10}(2008), 261-307.
		
		\bibitem{TaoVu09} T. Tao and V. Vu. Random matrices: Universality of local spectral statistics of non-Hermitian matrices. \emph{Ann. Probab.,} \textbf{38}(5)(2009), 2023-2065. 
		\bibitem {Temme} Temme, N. M. (1996). 
Asymptotic Expansions for the Incomplete Gamma Functions. \emph{
SIAM J. Math. Anal.}, {\bf 27}(1), 20-36.

		
		
	\end{thebibliography}
\end{document}